\documentclass[12pt]{amsart}
\usepackage{amsmath,amssymb,amsbsy,amsfonts,amsthm,latexsym,
  amsopn,amstext,amsxtra,euscript,amscd,stmaryrd,mathrsfs,
  cite,array,mathtools,enumerate, bm}
\usepackage{url}
\usepackage[colorlinks,linkcolor=blue,anchorcolor=blue,citecolor=blue,backref=page]{hyperref}
\usepackage{color}
\usepackage{graphics,epsfig}
\usepackage{graphicx}
\usepackage{float} 
\usepackage[english]{babel}
\usepackage{mathtools}
\usepackage{todonotes}
\usepackage{url}
\usepackage{comment}

\usepackage[norefs,nocites]{refcheck}

\hypersetup{breaklinks=true}

\usepackage[margin=1.5in,footskip=0.25in]{geometry}

\usepackage[norefs,nocites]{refcheck}
\usepackage[english]{babel}
\begin{document}

\newtheorem{thm}{Theorem}
\newtheorem{lem}[thm]{Lemma}
\newtheorem{claim}[thm]{Claim}
\newtheorem{cor}[thm]{Corollary}
\newtheorem{prop}[thm]{Proposition} 
\newtheorem{definition}[thm]{Definition}
\newtheorem{rem}[thm]{Remark} 
\newtheorem{question}[thm]{Question}
\newtheorem{conj}[thm]{Conjecture}
\newtheorem{prob}{Problem}
\newtheorem{lemma}[thm]{Lemma}

\theoremstyle{remark}
\newtheorem{rmk}[thm]{Remark}

\newcommand{\GL}{\operatorname{GL}}
\newcommand{\SL}{\operatorname{SL}}
\newcommand{\lcm}{\operatorname{lcm}}
\newcommand{\ord}{\operatorname{ord}}
\newcommand{\Op}{\operatorname{Op}}
\newcommand{\Tr}{\operatorname{Tr}}
\newcommand{\Nm}{\operatorname{Nm}}
\newcommand{\discr}{\operatorname{disrc}}
\newcommand{\Ker}{\operatorname{Ker}\,}
\renewcommand{\Im}{\operatorname{Im}}
\renewcommand{\Re}{\operatorname{Re}}

\numberwithin{equation}{section}
\numberwithin{thm}{section}
\numberwithin{table}{section}

\def\sssum{\mathop{\sum\!\sum\!\sum}}
\def\ssum{\mathop{\sum\ldots \sum}}
\def\iint{\mathop{\int\ldots \int}}
\def\hh{h} 

\def\vol {{\mathrm{vol\,}}}
\def\squareforqed{\hbox{\rlap{$\sqcap$}$\sqcup$}}
\def\qed{\ifmmode\squareforqed\else{\unskip\nobreak\hfil
\penalty50\hskip1em\null\nobreak\hfil\squareforqed
\parfillskip=0pt\finalhyphendemerits=0\endgraf}\fi}

\def \balpha{\bm{\alpha}}
\def \bbeta{\bm{\beta}}
\def \bgamma{\bm{\gamma}}
\def \blambda{\bm{\lambda}}
\def \bchi{\bm{\chi}}
\def \bphi{\bm{\varphi}}
\def \bpsi{\bm{\psi}}
\def \bomega{\bm{\omega}}
\def \btheta{\bm{\vartheta}}

\newcommand\veca{\boldsymbol{a}}
\newcommand\vecb{\boldsymbol{b}}
\newcommand\vech{\boldsymbol{h}}
\newcommand\vecq{\boldsymbol{q}}
\newcommand\vecu{\boldsymbol{u}}
\newcommand\vecv{\boldsymbol{v}}
\newcommand\vecw{\boldsymbol{w}}
\newcommand\vecx{\boldsymbol{x}}
\newcommand\vecy{\boldsymbol{y}}
\newcommand\vecz{\boldsymbol{z}}

\newcommand{\bfxi}{{\boldsymbol{\xi}}}
\newcommand{\bfrho}{{\boldsymbol{\rho}}}

\def\sfB{\mathsf {B}}
\def\sfJ{\mathsf {J}}
\def\sfG {\mathsf {G}}
\def\sfK {\mathsf {K}}
\def\sfS {\mathsf {S}}

\def\sL {\mathscr  {L}}

 \def \xbar{\overline x}
  \def \ybar{\overline y}

\def\cA{{\mathcal A}}
\def\cB{{\mathcal B}}
\def\cC{{\mathcal C}}
\def\cD{{\mathcal D}}
\def\cE{{\mathcal E}}
\def\cF{{\mathcal F}}
\def\cG{{\mathcal G}}
\def\cH{{\mathcal H}}
\def\cI{{\mathcal I}}
\def\cJ{{\mathcal J}}
\def\cK{{\mathcal K}}
\def\cL{{\mathcal L}}
\def\cM{{\mathcal M}}
\def\cN{{\mathcal N}}
\def\cO{{\mathcal O}}
\def\cP{{\mathcal P}}
\def\cQ{{\mathcal Q}}
\def\cR{{\mathcal R}}
\def\cS{{\mathcal S}}
\def\cT{{\mathcal T}}
\def\cU{{\mathcal U}}
\def\cV{{\mathcal V}}
\def\cW{{\mathcal W}}
\def\cX{{\mathcal X}}
\def\cY{{\mathcal Y}}
\def\cZ{{\mathcal Z}}

\def\NmQR{N(m;Q,R)}
\def\VmQR{\cV(m;Q,R)}

\def\Xm{\cX_m}

\def \A {{\mathbb A}}
\def \B {{\mathbb A}}
\def \C {{\mathbb C}}
\def \F {{\mathbb F}}
\def \G {{\mathbb G}}
\def \L {{\mathbb L}}
\def \K {{\mathbb K}}
\def \N {{\mathbb N}}
\def \PP {{\mathbb P}}
\def \Q {{\mathbb Q}}
\def \R {{\mathbb R}}
\def \Z {{\mathbb Z}}

\def \fU{\mathfrak U}
\def \fS{\mathfrak S}
\def \fM{\mathfrak M}
\def \fE{\mathfrak E}
\def \fB{\mathfrak B}
\def \fA{\mathfrak A}

\def\GL{\operatorname{GL}}
\def\SL{\operatorname{SL}}
\def\PGL{\operatorname{PGL}}
\def\PSL{\operatorname{PSL}}
\def\li{\operatorname{li}}
\def\sym{\operatorname{sym}}

\def\Mob{M{\"o}bius }

\def\fF{\mathfrak{F}}
\def\M{\mathsf {M}}
\def\T{\mathsf {T}}

\def\e{{\mathbf{\,e}}}
\def\ep{{\mathbf{\,e}}_p}
\def\eq{{\mathbf{\,e}}_q}

\def\fF{\EuScript{F}}
\def\M{\mathsf {M}}
\def\T{\mathsf {T}}

\def\sR{{\mathscr R}}

\def\\{\cr}
\def\({\left(}
\def\){\right)}
\def\fl#1{\left\lfloor#1\right\rfloor}
\def\rf#1{\left\lceil#1\right\rceil}

\def \oF {\overline \F}

\newcommand{\pfrac}[2]{{\left(\frac{#1}{#2}\right)}}
\newcommand{\Mod}[1]{\ (\mathrm{mod}\ #1)}

\def \Prob{{\mathrm {}}}
\def\e{\mathbf{e}}
\def\ep{{\mathbf{\,e}}_p}
\def\epp{{\mathbf{\,e}}_{p^2}}
\def\em{{\mathbf{\,e}}_m}

\def\Res{\mathrm{Res}}
\def\Orb{\mathrm{Orb}}

\def\vec#1{\mathbf{#1}}
\def \va{\vec{a}}
\def \vb{\vec{b}}
\def \vc{\vec{c}}
\def \vh{\vec{h}}
\def \vm{\vec{h}}
\def \vs{\vec{s}}
\def \vu{\vec{u}}
\def \vv{\vec{v}}
\def \vx{\vec{x}}
\def \vy{\vec{y}}
\def \vz{\vec{z}}

\def\flp#1{{\left\langle#1\right\rangle}_p}
\def\T {\mathsf {T}}

\def\mand{\qquad\mbox{and}\qquad}

\def \l {{\lambda}}
\def \a {{\alpha}}
\def \b {{\beta}}
\def \f {{\phi}}
\def \r {{\rho}}
\def \R {{\mathbb R}}
\def \H {{\mathbb H}}
\def \N {{\mathbb N}}
\def \C {{\mathbb C}}
\def \Z {{\mathbb Z}}
\def \Q {{\mathbb Q}}
\def \e {{\epsilon }}
\def\GL{\ensuremath{\mathop{\textrm{\normalfont GL}}}}
\def\SL{\ensuremath{\mathop{\textrm{\normalfont SL}}}}
\def\Gal{\ensuremath{\mathop{\textrm{\normalfont Gal}}}}
\def\SU{\ensuremath{\mathop{\textrm{\normalfont SU}}}}
\def\SO{\ensuremath{\mathop{\textrm{\normalfont SO}}}}

\title[Large values of $L(\sigma,\chi)$ over sets of characters]{Large values of $L(\sigma,\chi)$ over sets of characters with small product set}

\author[P. Darbar] {Pranendu Darbar}
\address{School of Mathematics and Statistics, University of New South Wales, Sydney NSW 2052, Australia}
\email{darbarpranendu100@gmail.com}

 \author[B. Kerr] {Bryce Kerr}
\address{School of Science, University of New South Wales,
Canberra, ACT 2106, Australia}
\email{bryce.kerr@unsw.edu.au}

\author[M. Munsch]{Marc Munsch}
\address{Institut Camille Jordan, University of  Jean Monnet,
20, rue du Docteur Rémy Annino, 42 000 Saint-Etienne Cedex 2, France}
\email{marc.munsch@univ-st-etienne.fr}

\author[I. E. Shparlinski] {Igor E. Shparlinski}
\address{School of Mathematics and Statistics, University of New South Wales, Sydney NSW 2052, Australia}
\email{igor.shparlinski@unsw.edu.au}

 \date{\today}

\begin{abstract}  We obtain (conditional and unconditional) results on large values of $L$-functions 
$L(s,\chi)$ in the critical strip $1/2 \leq \Re s \leq 1$ when 
the character $\chi$ runs through the ratio set $\cA/\cA$ of a thin set $\cA$ of characters 
with small product set modulo a prime $q$.  Some of these bounds are based on new zero-density 
estimates on average over sets of characters with a small product set. These bounds follow from a 
mean value estimate for character sums, which is based on the work of 
D.~R.~Heath-Brown (1979). 
As yet another application of this mean value estimate, we obtain an unconditional version 
of a conditional (on the Generalised Riemann Hypothesis) result of  
Z.~Rudnick and A.~Zaharescu (2000) about gaps between primitive roots.  
 \end{abstract}


\maketitle

\tableofcontents

\section{Introduction}
\subsection{Motivation and set-up}
A classical problem in analytic number theory is to understand the distribution of values of Dirichlet $ L$-functions 
\[
L(s,\chi) = \sum_{n=1}^\infty \frac{\chi(n)}{n^s}, \qquad \Re s > 0, 
\]
corresponding to non-principal multiplicative characters modulo $q$ 
in the critical strip $\Re s \in (0,1]$; see~\cite[Chapter~3]{IwKow}
for background on characters. 

On one hand, there are numerous works about the value distribution of $L(s,\chi)$ when the  characters $\chi$  run through the  
full group $\cX_q$ of  multiplicative characters modulo $q$; 
 see, for example~\cite{AMMP, ArCr, BKT, BFKMM1, BFKMM2, DWZ, GrSo-3, KMN2, Lam0, Lam1, LeLe, Sound, Wu1, Wu2} and references therein.
Furthermore, there are several recent works where the values of $L(s,\chi)$ are studied over some partial sets 
of $\chi$, most commonly characters of fixed small order~\cite{DDLL, DM2025, GrSo-2, HeLi, Lam2} or characters running through a subgroup or a coset of $\cX_q$; see~\cite{Erm, FKM, 
GaYo, KLM, KMN1, Lou1, Lou2, LoMu1, LoMu2, MunShp}.

Here we continue this line of research and study the large values of  $L(s,\chi)$ when $\chi$ varies over the ratio
set of a rather thin set of characters with small multiplicative doubling. 
For instance, this includes the case when 
$\chi$ runs through a subgroup of $\cX_q$, or, more generally, sets $\cA$ which are intervals of the form 
\begin{equation}
\label{eq:Example A}
\cA = \{\chi_*^j:~j =1, \ldots, A\}
\end{equation}
for some positive integer $A \le m$, where $m$ is the order of some fixed character $\chi_*\in \cX_q$.
In particular, in Section~\ref{sec: extrval},  we investigate three different regimes of real values of $s$:
\begin{equation}
\label{eq:ranges}
s = 1, \qquad  1> s > 1/2, \qquad s = 1/2,
\end{equation}
using the resonance method combined with different arguments depending on the regime.  Furthermore, in Section~\ref{sec:zerodens}, we also obtain a new zero-density estimate 
for $L$-functions averaged over sets of characters with small product set, which is required for the second case in~\eqref{eq:ranges}.  In turn, this estimate rests on a new bound on the mean value of character sums over sets of characters with small product set, which improves and generalises that of Montgomery~\cite{Mont2}; see Section~\ref{sec:meanvalchar}. This result has several independent applications. In particular, we  improve some previous results of Montgomery~\cite[Theorem~2]{Mont2} and of 
Rudnick and Zaharescu~\cite[Theorem~3(i)]{RuZa}  for  a question 
related to the distribution of gaps between the powers of primitive roots; see Section~\ref{sec:Appl} for details.

We emphasise that the sequence in which we prove our results deviates from that in which we formulate them.  In particular, we start with proving Theorem~\ref{thm:MeanVal}, from which we derive Theorem~\ref{thm:mean-value-sub}, and at the end of the paper we prove the results at the beginning of 
Section~\ref{sec: extrval} on extreme values of $L$-functions. 

\subsection{Notation and conventions}

As usual, the notations  $A \ll B$, $B \gg A$ and $A = O(B)$, are equivalent
to $|A|  \le c B$ for some constant $c>0$,
which throughout this work may sometimes, where obvious, 
depend on the real positive parameters $\delta$, $\varepsilon$, $\sigma$ and $K$, as well as on some other parameters 
declared to be fixed, and are absolute otherwise.   

We use $\Z_q$ to denote the residue ring modulo an integer $q \ge 1$ and
we use  $\Z_q^*$ to denote its group of units. For a prime $q$, we use 
$\F_q$ and $\F_q^*$ instead of $\Z_q$ and $\Z_q^{*}$, respectively. We also let $\cX_q$ denote the group of multiplicative characters mod $q$.

For a finite set $\cS$ we use $\#\cS$  to denote its cardinality. 

For an integer $\nu \ge 1$ and $z \ge 1$, we define 
$\log_\nu z$ recursively by $\log_1 z = \max\{2,  \log z\}$ and 
for $\nu\ge 2$ by $\log_\nu z = \max\{2,  \log_{\nu-1} z\}$. 

We always use $\gamma_0 = 0.5772\ldots$ to denote the 
Euler--Mascheroni constant. 

As usual, we abbreviate the Generalised Riemann Hypothesis as GRH.

\section{Main results}

\subsection{Extreme values}
\label{sec: extrval}

For a set $\cA\subseteq \cX_q$, we write
\[
\cA^2=\{\chi_1\chi_2:~\chi_1,\chi_2\in \cA\}, \qquad
\cA^{-1}=\{\overline\chi:~\chi\in \cA\},
\]
and define  $\cA/\cA$ n the natural way.

Our first result deals with the case $s = 1$.

\begin{thm}\label{th1}
	Assume that $q$ is a sufficiently large prime and let $\delta > 0$ and let $K\ge 1$ be fixed.  
Let $\cA\subseteq \mathcal{X}_q$ be a set of cardinality $A=\#\cA$ satisfying 
\begin{itemize}
\item[(i)]  $A\geq  (\log q)^{1+\delta}$ unconditionally; 
\item[(ii)] $A\geq  (\log_2 q)^{1+\delta}$,  assuming the GRH. 
\end{itemize}
	There exists a non-principal character $\chi \in \cA/\cA$ such that 
\[
 \left|L(1,\chi)\right|  \ge e^{\gamma_0}\!\(\log_2 A  +\log_3 A - \log_2  4 - \log \( 1+\delta^{-1}\)- 1  + o(1)\).  
\] \end{thm}  

We note that when $\delta \to \infty$ (which is an admissible regime in Theorem~\ref{th1}) we have 
$\log \( 1+\delta^{-1}\) \to 0$.

Next, we consider real values $s = \sigma \in (1/2,1)$ inside the critical strip. In this case our results apply to 
sets of characters with small multiplicative doubling.
That is,  we show that there are large values of $L(s,\chi)$ 
with $\chi\in \cA/\cA$ if $\cA$ is a thin set with small product set. Assume that 
\begin{equation}
\label{eq:Aprodbound}
\#\cA^2\le KA,
\end{equation}
where $A =\# \cA$. 
Then by the Pl{\"u}nnecke--Rusza inequality; see~\cite[Corollary~6.29]{TaoVu},
\begin{equation}
\label{eq:Plunnecke-second-difference}
\#\(\cA/\cA\)\le K^{2}A.
\end{equation}
Hence our result below produces large values of $L(\sigma,\chi)$ in sets of size  $K^{2}A$. The most natural example of such sets are subgroups, since if $\cA$ is a subgroup then $\cA/\cA=\cA$. The advantage of working in greater generality than subgroups is the following. Our main results only apply to sets $\cA$ of specific size. The existence of subgroups of desired sizes depends on the factorisation of $q-1$ into primes.  
 However,  the sets of the shape~\eqref{eq:Example A} exist for any $q$ and any $A \le q-1$.

Recall that the implied constants may depend on $\delta$, $\sigma$ and $K$.

 \begin{thm}\label{th3} Assume that $q$ is a sufficiently large prime and let $1/2<\sigma<1$, $0 < \delta < 1$ and $K\ge 1$ be fixed. 
  Let $b(\sigma)=8(1-\sigma)/(2\sigma-1)$ and let  $\cA\subseteq \mathcal{X}_q$ be a set of cardinality $A=\#\cA$ satisfying~\eqref{eq:Aprodbound} 
  and
  \begin{itemize}
\item[(i)]  $A\geq q^{(2-2\sigma)/(2-\sigma) + \delta}$ unconditionally; 
\item[(ii)] $A  \geq (\log q)^{(1+\delta) b(\sigma)}$,  assuming the GRH. 
\end{itemize}
There exists a non-principal character $\chi \in \cA/\cA$ such that
\[
\log \left|L(\sigma,\chi)\right| \gg (\log A)^{1-\sigma} (\log_2 A)^{-\sigma},
\]
as $A\to \infty$. 
\end{thm} 

\begin{rmk}
The proof of Part~(i) of Theorem~\ref{th3} is based on a new zero density estimate provided by Theorem~\ref{thm:mean-value-sub} ,which is better than the trivial bound~\eqref{eq: TrivBound} as soon as $A \geq q^{(2-2\sigma)/(2-\sigma)+\delta}$ for any $\delta> 0$. Note that the constants involved in the lower bound depend only on $\sigma$, $\delta$ and $K$ and could be computed effectively. To allow all values of $\sigma \in (1/2,1)$, one requires $A \geq q^{2/3+\delta}$ unconditionally.
If $A$ falls below this threshold, the admissible range of $\sigma$ shrinks from the left.
\end{rmk}

Note that for intervals $\cA$ as in~\eqref{eq:Example A}, since 
$\cA/\cA= \cA \cup \{\overline\chi:~\chi\in \cA\} \cup \{\chi_0\}$ where $\chi_0$ is the trivial character and $\left|L(\sigma,\chi)\right| = \left|L(\sigma,\overline\chi)\right|$, 
we can choose $\chi$ in  Theorem~\ref{th3} from the original set $\cA$ rather than from $\cA/\cA$.

We now examine the behaviour at the critical point $s=1/2$.

\begin{thm}\label{th4} 
Assume that $q$ is a sufficiently large prime and let $K\ge 1$ be fixed. Suppose that $\cA\subseteq \cX_q$ 
is a set of cardinality $A = \#\cA$.   
If  $A\geq q^{1/2+\varepsilon}$  for some fixed $\varepsilon > 0$, then there exists a non-principal 
character $\chi \in  \cA/\cA$ such that  
	\[
\left|L\(\tfrac{1}{2}, \chi\)\right|\gg \left(\frac{A}{q}\right)^{1/2}
	\exp\((\sqrt{2}+o(1))\sqrt{\tfrac{\(\log \(A/\sqrt{q}\)\)\log_3 A }{\log_2 A}}\).
	\]
\end{thm}

\begin{rmk} We can have a version of Theorem~\ref{th4} for smaller values of $A$. 
However, in order to obtain a large value of Bondarenko--Seip type (see~\cite{BoSe}), that is, to guarantee that the exponential 
factor dominates $(A/q)^{1/2}$ in the bound of Theorem~\ref{th4}, it is sufficient to ensure that  
\[
\frac{q}{A} \leq \exp \( 2(1-c)\sqrt{\tfrac{(\log q)\log_3 q}{\log_2 q}}\),
\]  
for some $0<c<1$. This leads to the size condition 
\[A>q\exp \left(-2(1-c)\sqrt{\tfrac{(\log q)\log_3 q}{\log_2 q}}\right).
\]
A typical example is given by choosing $c=1/2$ and thus
\[
A = q e^{-(1+o(1))\sqrt{\tfrac{(\log q) \log_3 q}{\log_2 q}}}
=q^{1-(1+o(1))\sqrt{\tfrac{\log_3 q}{(\log q) \log_2 q}}}, 
\]
for which Theorem~\ref{th4} yields 
\[
\max_{\substack{\chi\in \cA/\cA\\chi\neq \chi_0}}\left|L\(\tfrac{1}{2}, \chi\)\right| \ge 
\exp\(\(\frac{1}{2}+o(1)\)\sqrt{\tfrac{(\log q) \log_3 q}{\log_2 q}} \).
\]   
Furthermore, in the subgroup framework, one can attain a bound of the same strength 
as that of Soundararajan~\cite{Sound} by selecting a slightly larger value of $d$, which in turn results in a smaller subgroup size; see~\cite{Erm}, where this approach has been implemented. 
\end{rmk}

\subsection{Zero density estimates over sets of characters with small doubling}
\label{sec:zerodens}
We now present a new zero density estimate on average over sets of characters with small doubling. 
The estimate is used in the proof of Theorem~\ref{th3}, but 
we believe that a more general result is of independent interest. 

 Let $N(\sigma,T, \chi)$ denote the number of zeros of $L(s,\chi)$ such that $\Re s\geq \sigma$ and $|\Im s|\leq T$.

\begin{thm}
\label{thm:mean-value-sub} 
Assume that $q$ is a sufficiently large prime and let $\sigma > 1/2$.  
 Let $\cA\subseteq \mathcal{X}_q$ be a set of cardinality $A=\#\cA$ satisfying~\eqref{eq:Aprodbound}
 for some $K\ge 1$.
   For any $T\le (\log{q})^{O(1)}$
\[
\sum_{\chi\in \cA}N(\sigma,T; \chi)\le  q^{o(1)} K^{1/2}
 \begin{cases} A^{(7-6\sigma)/(6-4\sigma)}& \text{if } A \ge  q^{2/3},\\
 A^{(4-3\sigma)/(6-4\sigma)} q^{(1-\sigma)/(3-2\sigma)} & \text{if } A < q^{2/3}.
 \end{cases} 
\]
\end{thm} 

Note that $K$ in Theorem~\ref{thm:mean-value-sub}  is not fixed and is allowed to grow with $q$.

We recall the classical bound  
\begin{equation}  
 \label{eq: N(1,T)}
N(T; \chi)= N(0,T; \chi)\ll T \log(qT)+\log{q}; 
\end{equation}  
see, for example~\cite[Chapter~16]{Dav}  or~\cite[Theorem~5.8]{IwKow}. Hence we trivially have 
\begin{equation}  
 \label{eq: TrivBound}
\sum_{\chi\in \cA}N(\sigma,T; \chi)\le A  q^{o(1)} 
\end{equation}
in the range of $T$ of Theorem~\ref{thm:mean-value-sub}. On the other hand, one easily verifies 
that   
\[
\frac{7-6\sigma}{6-4\sigma} < 1
\]
for $\sigma > 1/2$. Hence, if $K \ll 1$ (as for subgroups $\cA$ and sets of the form~\eqref{eq:Example A}) 
and $A \ge  q^{2/3}$,  then Theorem~\ref{thm:mean-value-sub} is 
nontrivial for any fixed $\sigma>1/2$. A similar computation shows that for  $K \ll 1$ and $q^{2/3} >  A \ge q^\eta$
with some fixed $\eta>0$,
Theorem~\ref{thm:mean-value-sub} is  stronger than~\eqref{eq: TrivBound} as long as 
\[
\sigma > \frac{2-2\eta}{2-\eta}.
\]

The proof of Theorem~\ref{thm:mean-value-sub} is based on a new bound on some mean values 
of character sums given below which builds on an estimate of Heath-Brown~\cite[Theorem~1]{H-B}.

\subsection{Mean values of character sum over sets of characters with small multiplicative doubling}
\label{sec:meanvalchar}
For an integer $N$ and a set $\cA\subseteq \cX_q$ of cardinality $A = \# \cA$, we consider the mean value 
\[
M(\cA, N) =\frac{1}{A} \sum_{\chi \in \cA} \left| \sum_{n=1}^N \chi(n)\right|. 
\]

We recall that  in the case of intervals  $\cA$ as in~\eqref{eq:Example A}  of size $A$
the mean value $M(\cI, N)$ has been estimated 
by Montgomery~\cite[Lemma~5]{Mont2} as
\begin{equation}\label{eq:MontBound}  
M(\cA, N) \le \(A^{-1/4} N + A^{-1/8} q^{1/4} +    q^{1/8} + Nq^{-1/8}\) q^{o(1)}, 
\end{equation}  
provided $N \le (q/2)^{1/2}$. 

Here we show that a result of Heath-Brown~\cite[Theorem~1]{H-B}  leads to an improvement 
of~\eqref{eq:MontBound}. In fact, it can be combined with  {\it Green--Ruzsa covering lemma\/}
(in a form given by~\cite[Lemma~2.17]{TaoVu}) to apply to arbitrary sets
with small multiplicative doubling as in~\eqref{eq:Aprodbound}.

We state the following improvement of~\eqref{eq:MontBound} in greater generality. To do so, we  introduce complex coefficients $\balpha = (\alpha_n)_{n =1}^N$ 
in the character sums and define 
\[
M(\balpha, \cA, N) = \frac{1}{A}  \sum_{\chi \in \cA} \left| \sum_{n=1}^N \alpha_n \chi(n)\right|, 
\]
where, as before,  $A = \#\cA$. 

\begin{thm}
	\label{thm:MeanVal} Assume that $q$ is a sufficiently large prime. 
	 Let $K \ge 1$ be fixed and  
 let $\cA\subseteq \mathcal{X}_q$ be a set of cardinality $A=\#\cA$ satisfying~\eqref{eq:Aprodbound}.
Then, for a positive integer $N \le q$ and complex coefficients $|\alpha_n| \le 1$, $1 \le n \le N$, we have   
	\[
	M(\balpha, \cA, N) \le K^{1/2} \(N^{1/2}  + A^{-1/2} N+  A^{-3/8} N^{1/2}  q^{1/4}\) q^{o(1)}.
	\]
\end{thm}

We observe that for either $A> q^{2/3}$ or $N\ge q^{2/3}$ the bound of Theorem~\ref{thm:MeanVal} simplifies to
\[
M(\balpha, \cA, N) \le K^{1/2} \(N^{1/2}  + A^{-1/2} N\) q^{o(1)}.
\]

The bound~\eqref{eq:MontBound}   has been used first in the same work of Montgomery~\cite[Theorem~2]{Mont2}
and then also  by Rudnick and Zaharescu~\cite[Theorem~3(i)]{RuZa}  for  a question 
related to the distribution of gaps between the powers of primitive roots.

Substituting this new bound of Theorem~\ref{thm:MeanVal}  in the arguments of Montgomery~\cite{Mont2}
and of Rudnick and Zaharescu~\cite{RuZa}, we obtain stronger versions 
of~\cite[Theorem~2]{Mont2} and~\cite[Theorem~3(i)]{RuZa}. Our improvement 
makes~\cite[Theorem~3(i)]{RuZa} unconditional by removing the GRH assumption; see Section~\ref{sec:Appl} for details.

\begin{rmk} It is easy  to see that by the Cauchy--Schwarz 
inequality and the orthogonality 
relation between characters from $\cG$, the question of estimating $M(\cG, N)$ can be reduced to counting 
solutions to the congruence $m \equiv \gamma n \bmod q$ with integers $1 \le m,n \le N$ 
and elements $\gamma\in \Gamma$ from some subgroup $\Gamma \subseteq \Z_q^*$ 
(which is dual to $\cG$ and thus is of order $\# \Gamma = \varphi(q)/\# \cG$).  An upper bound on 
the number of solutions to such congruences is given in~\cite[Theorem~1]{BKS}; 
see also~\cite[Theorem~1]{CillGar} and~\cite[Theorem~2]{Shp}. Perhaps in some cases, especially for very small subgroups this approach can lead to a better result, but it is not clear whether these results lead to any interesting applications to $L$-functions.
\end{rmk}

\section{Preliminaries}

 \subsection{Double sums of character sums}

 The following bound is a simplified version of~\cite[Theorem~1]{H-B} (below we appeal to the notation of~\cite{H-B}). We note that in 
 Lemma~\ref{lem:H-B} below all characters are distinct and of the same conductor $q$. This 
  allows us to choose $t_r = 0$ (which is an admissible choice in this case)  and thus $T =1$, $Q = q_0 = q$, 
 hence $D =q$. Besides we choose the weights $a_n = n^{1/2}$, $1 \le n \le N$, 
 (in the notation of~\cite{H-B}). We can  now record the corresponding 
 version of~\cite[Theorem~1]{H-B} as follows.

 \begin{lemma}
 	\label{lem:H-B}
 	Let $\chi_r$, $r = 1, \ldots, R$, be $R\ge 1$ distinct multiplicative characters of $\F_q$. 
 	Then, for a positive integer $N \le q$ and complex coefficients $|\alpha_n| \le 1$, $1 \le n \le N$, we have   
 	\[
 	\sum_{r,s =1}^R  \left| \sum_{n=1}^N \alpha_n \chi_r(n)\overline{\chi_s}(n)\right|^2
 	\le \(N^2R + NR^2 + NR^{5/4} q^{1/2}\) q^{o(1)}. 
 	\]
 \end{lemma}
 
 \subsection{Green--Rusza covering lemma}
 
 We need the following special case (of equal sets $\cA = \cB \subseteq \cX_q$) provided by Tao and Vu~\cite[Lemma~2.17]{TaoVu} and attributed to Green and Rusza~\cite{GrRu}.

 \begin{lemma}
 	\label{lem:G-R Cover} Let $K\ge 1$ be fixed and let $\cA\subseteq \mathcal{X}_q$ be a set of cardinality $A=\#\cA$ satisfying~\eqref{eq:Aprodbound}  Then there exists 
 	a set $\cU  \subseteq \cX_q$ of cardinality $\# \cU \le 2K-1$ such that 
 	any  character  $\chi\in \cA$ has at least $A/2$ representations as 
 	\[
 	\chi(n) = \chi_1(n) \overline\chi_2(n) \eta(n) , \qquad \chi_1, \chi_2\in \cA, \  \eta \in \cU.
 	\]
 \end{lemma}

\subsection{Approximation and bounds for  some finite Euler products and    $L$-functions}

The following is an effective version of~\cite[Equation~(2.13)]{AMMP}.
\begin{lemma}
\label{lem:q2} 
For $X\gg 1$ and $p\le X$, let 
\[r_p=\left(1-\frac{p}{X}\right).\]
Then 
\[\log \prod_{p\le X}\left(1-r_p^2 \right)^{-1} =(2-\log{4})\frac{X}{\log X}+O\left( \frac{X  \log_2 X }{(\log X)^2}\right).
\]
\end{lemma} 

\begin{proof}
This is based on some manipulations from~\cite{AMMP}. By partial summation and the Prime Number Theorem, we have
\begin{align*}
\log \prod_{p\le X}\left(1-r_p^2 \right)^{-1}&=2\int_{2}^{X}\frac{\pi(x)(X-x)}{x(2X-x)}dx \\
&=2\int_{2}^{X}\frac{(X-x)}{(\log X)(2X-x)}dx+O\left(\frac{X}{(\log X)^2} \right), 
\end{align*}
where as usual, $\pi(x)$ is the number of primes $p\le x$. 
After some more elementary manipulations
\begin{align*}
\log \prod_{p\le X}&\left(1-r_p^2 \right)^{-1} =\frac{2X}{\log X}-2X\int_{2}^{X}\frac{1}{(2X-x)\log x}dx+O\left(\frac{X}{(\log X)^2}\right) \\ 
&=\frac{2X}{\log X}-2X\int_{1}^{2-2/X}\frac{1}{t(\log{(2-t)}+\log X)}dt+O\left(\frac{X}{(\log X)^2}\right) \\ 
& =(2-2\log{2})\frac{X}{\log X}+O\left( \frac{X  \log_2 X}{(\log X)^2}\right).
\end{align*} 
Indeed, let $\Delta = (\log X)^{-2}$ and split the last integral 
\[\int_{1}^{2-2/X}\frac{1}{t(\log{(2-t)}+\log X)}dt = J_1 + J_2, 
\]
where 
\begin{align*}
 & J_1 = \int_{1}^{2-\Delta}\frac{1}{t(\log{(2-t)}+\log X)}dt, \\
 & J_2 = \int_{2-\Delta }^{2-2/X}\frac{1}{t(\log{(2-t)}+\log X)}dt. 
\end{align*}
 For $J_1$ we write 
\begin{align*}
 J_1 & = \frac{1}{\log X + O\(\log\Delta^{-1}\)}  \int_{1}^{2-\Delta}\frac{1}{t}dt =   \frac{1}{\log X + O\(\log\Delta^{-1}\)} \log (2-\Delta) \\
  & =   \frac{\log 2+ O(\Delta) }{\log X + O\(\log\Delta^{-1}\)} 
  =  \frac{\log 2}{\log X }  +O\left( \frac{  \log_2 X}{(\log X)^2}\right).
  \end{align*}

Furthermore, we also have 
\begin{align*}
J_2  & \ll  \int_{2-\Delta }^{2-2/X}\frac{1}{ \log{(2-t)}+\log X }dt
\ll   \int^{\Delta}_{2/X}\frac{1}{\log t +\log X}dt \\
& =  \int^{\Delta}_{2/X}\frac{1}{\log (t X)}dt  = X^{-1}   \int^{X\Delta}_{2}\frac{1}{\log u}d u\\
& \ll  X^{-1} \frac{X\Delta}{\log (X\Delta)} \ll (\log X)^{-3}. 
  \end{align*}
We now arrive at the desired result. 
\end{proof}

We require the following consequence of the approximate functional equation for Dirichlet $L$-functions.

\begin{lemma} 
\label{lem:approx}
Let $q$ be an integer and $\chi$ be a primitive character mod $q$. For any $|t|\le q^{\varepsilon/2}$, we have 
\[
 L\left(\frac{1}{2}+it,\chi\right) \ll \left|\sum_{n\le q^{1/2+\varepsilon}}\frac{\alpha(n,\chi(-1),t)\chi(n)}{n^{1/2}}\right|+1
\]
for some complex numbers $\alpha(n,\chi(-1),t) \ll 1$.
\end{lemma}

\begin{proof}
We  appeal to~\cite[Theorem~5.3]{IwKow}, which states that for $X>0$,
\begin{align*}
L\left(\frac{1}{2}+it,\chi\right) =\sum_{n\geq 1}\frac{\chi(n)}{n^{1/2+it}}&V_{1/2+it}\left(\frac{n}{X\sqrt{q}},\chi\right)\\
 &   +\varepsilon(\chi,s)\sum_{n\geq 1}\frac{\overline{\chi}(n)}{n^{1/2-it}}V_{1/2-it}\left(\frac{nX}{\sqrt{q}},\chi\right),
\end{align*}
where:   
\begin{enumerate} 
\item $V_s(y,\chi)$ is a smooth function defined by 
\[
V_s(y,\chi)=\frac{1}{2\pi i}\int_{\Re u = 3}y^{-u}G(u)\frac{\gamma(\chi,s+u)}{\gamma(\chi,s)}\frac{du}{u}
\]
for an arbitrary even function  $G(u)$, which is holomorphic and bounded in $-4<\Re u<4$ and satisfies $G(0)=1$; 
\item  $\varepsilon(\chi)$ and $\varepsilon(\chi,s)$ are  defined by 
\[
\varepsilon(\chi)=\frac{\tau(\chi)}{\sqrt{q}} \mand 
\varepsilon(\chi,s)=\varepsilon(\chi)q^{1/2-s}\frac{\gamma(\chi,1-s)}{\gamma(\chi,s)}, 
\]
where $\tau(\chi)$ is the usual Gauss sum; 
\item $\gamma(\chi,s)$ is  defined by
\[\gamma(\chi,s)=\pi^{-s/2}\Gamma\left(\frac{s+\delta}{2}\right)
\]
with $\delta=0$ if $\chi$ is an even function and $\gamma(\chi,s) = 1$ otherwise. 
\end{enumerate}
We apply the above with $X=1$. Using the fact that if $\Re(s)=1/2$ then 
\[|\varepsilon(\chi,s)|=1
\]
we derive
\begin{equation}
\begin{split}
\label{eq:LLL111}
L\left(\frac{1}{2}+it,\chi\right) & \ll \left|\sum_{n\geq 1}\frac{\chi(n)}{n^{1/2+it}}V_{1/2+it}\left(\frac{n}{\sqrt{q}},\chi\right)\right|\\
& \qquad \qquad +\left|\sum_{n\geq 1}\frac{\overline{\chi}(n)}{n^{1/2-it}}V_{1/2-it}\left(\frac{n}{\sqrt{q}},\chi\right)\right|.
\end{split}
\end{equation}

We choose 
\[G(u)=\left(\cos\left(\frac{\pi u}{4A}\right) \right)^{-4A}
\]
for a sufficiently large integer $A$. By~\cite[Proposition~5.4]{IwKow}, we have 
\begin{equation}
\label{eq:VVit}
V_{1/2+it}\left(\frac{n}{\sqrt{q}},\chi\right)\ll \left(1+\frac{n}{\sqrt{q|t|}}\right)^{-A}.
\end{equation}

Substituting~\eqref{eq:VVit} into~\eqref{eq:LLL111}, using that $|t|\le q^{\varepsilon/2}$ and summing over $n$ results in 
\begin{align*}
L\left(\frac{1}{2}+it,\chi\right)& \ll \left|\sum_{n\le q^{1/2+\varepsilon}}\frac{\chi(n)}{n^{1/2+it}}V_{1/2+it}\left(\frac{n}{\sqrt{q}},\chi\right)\right|\\
& \qquad +\left|\sum_{n\le q^{1/2+\varepsilon}}\frac{\overline{\chi}(n)}{n^{1/2-it}}V_{1/2-it}\left(\frac{n}{\sqrt{q}},\chi\right)\right|+1, 
\end{align*}
from which the desired result follows, after another application of~\eqref{eq:VVit} and noting that $V_{1/2\pm it}\left(\frac{n}{\sqrt{q}},\chi\right)$ depends only on $n$, $t$ and the value of $\chi(-1)$. 
\end{proof}

We now recall the following approximation result for Dirichlet $L$-functions; see~\cite[Lemma 8.2]{GrSo-0}. 

\begin{lemma}\label{approxDir} Let $q$ be a large prime and let $\chi \in \cX_q$. Let  $X \geq 2$ and $|t|\leq 3q$ be real numbers.
Let $1/2 \leq \sigma_0 <\sigma\leq 1$ and suppose that the
rectangle 
\[\{ s:~\sigma_0 <\Re s \leq 1, \
|\Im s -t| \leq X+2\}\] 
does not contain any zeros of $L(s,\chi)$.
Then
\[
\log L(\sigma+it,\chi)= \sum_{2 \leq n \leq X} \frac{\Lambda(n)\chi(n)}{n^{\sigma+it}\log n} +
O\left( \frac{\log
q}{(\sigma_1-\sigma_0)^2}X^{\sigma_1-\sigma}\right),
\]
where  
\[\sigma_1 = \min\left\{\sigma_0+\frac{1}{\log X},
\frac{\sigma+\sigma_0}{2}\right\}.
\]
\end{lemma}

We require a variation of Lemma~\ref{approxDir},  which has better dependence on parameters for $\sigma=1$.

For a real number $X\ge 2$ we define
\[
L(1,\chi;X)=\prod_{p\le X}\left(1-\frac{\chi(p)}{p}\right)^{-1}.
\]

\begin{lemma}
\label{lem:L1approx} 
Let $q$ be a large integer and $0< c<2$. There exists a set $\cE$ of multiplicative characters $\chi \in \cX_q$ satisfying 
\begin{equation}
\label{eq:Ebound}
\#\cE\ll (\log{q})^{2c} (\log_2 q)^{3}
\end{equation}
such that if $\chi \not \in\cE$ then 
\[
L(1,\chi)=L(1,\chi;q^{4})\left(1+O\left(\frac{1}{(\log{q})^c}\right) \right). 
\]
\end{lemma}

\begin{proof}
We first note from~\cite[Equation~(2.1)]{AMMP} that 
\begin{equation} \label{lapprox-1}
L(1,\chi) = L(1,\chi;Z) \left(1 + O \left(\frac{1}{(\log q)^2} \right) \right),
\end{equation}
where 
\[Z=\exp\left((\log{q})^{20}\right).\]
Thus it is sufficient to show that there exists a set of characters $\cE$ satisfying~\eqref{eq:Ebound}  such that if $\chi\not \in \cE$ then 
\begin{equation}
\label{eq:LYZ}
\frac{L(1,\chi;Z)}{L(1,\chi;q^{4})}-1\ll \frac{1}{(\log{q})^{c}}.
\end{equation}
Therefore, if $\chi \not \in \cE$ then from~\eqref{lapprox-1},
\[
L(1,\chi)=L(1,\chi;q^{4})\left(1+O\left(\frac{1}{(\log{q})^{c}}\right)\right).
\]
Define 
\begin{equation}
\label{eq:Edef-1}
\cE= \left\{ \chi  :~\left|\sum_{q^{4}<p \le Z}\frac{\chi(p)}{p}\right|\ge (\log{q})^{-c} \right\}.
\end{equation}
Using
\begin{align*}
\frac{L(1,\chi;Z)}{L(1,\chi;q^4)}& =\exp\(\sum_{q^{4}<p\le Z}-\log\left(1-\frac{\chi(p)}{p} \right) \)\\
& =\exp\(\sum_{q^{4}<p\le Z}\frac{\chi(p)}{p}+O\left(\frac{1}{q}\right) \),
\end{align*}
we see that if $\chi \not \in \cE$ then~\eqref{eq:LYZ} holds. Hence it is sufficient to show that with $\cE$ as in~\eqref{eq:Edef-1} we have
\[
\#\cE\ll (\log{q})^{2c}(\log_2{q})^{3}.
\]
By Chebyshev's inequality,
\[
(\log{q})^{-c}\#\cE \leq \sum_{\chi \in \cE}\left|\sum_{q^{4}<p \le Z}\frac{\chi(p)}{p} \right| \leq \sum_{q^{4}<p \le Z}\frac{1}{p}\left|\sum_{\chi\in \cE}\vartheta(\chi)\chi(p)\right|
\]
for some complex numbers $\vartheta(\chi)$ satisfying $|\vartheta(\chi)|=1$. 
By the Cauchy--Schwarz inequality and the Mertens theorem (see~\cite[Equation~(2.15)]{IwKow}), we have 
\begin{align*}
(\log{q})^{-2c}(\#\cE)^2& \ll \log\left(\frac{\log{Z}}{\log{q}}\right)\sum_{q^{4}<p \le Z}\frac{1}{p}\left|\sum_{\chi\in \cE}\vartheta(\chi)\chi(p)\right|^2 \\ 
& \ll \(\log_2 q\) \cdot \sum_{q^{4}<p \le Z}\frac{1}{p}\left|\sum_{\chi\in \cE}\vartheta(\chi)\chi(p)\right|^2.
\end{align*}
Define the sequence $Y_i$ by 
\begin{equation}
\label{eq:Yichoice}
Y_0 = q^{4} \mand Y_{i} = Y_{i-1}^2, \quad i =1, 2, \ldots, 
\end{equation} 
so that after partitioning summation over $p$ into intervals determined by $Y_i$'s and recalling the choice of $Z$, we arrive at the
inequality
\begin{equation}
\label{eq:EESI}
(\log{q})^{-2c}(\#\cE)^2\ll\(\log_2 q\) \cdot  \sum_{\substack{2^{i}\le (\log{q})^{19}}}S_i
\ll \(\log_2 q\)^2 \cdot  \max_{\substack{2^{i}\le (\log{q})^{19}}}S_i, 
\end{equation}
where
\[
S_i=\sum_{Y_i<p\le Y_{i+1}}\frac{1}{p}\left|\sum_{\chi\in \cE}\vartheta(\chi)\chi(p)\right|^2.
\]
Define $y_i$ by 
\begin{equation}
\label{eq:yichoice}
y_i=\exp\left(\frac{\log{Y_i}}{C\log_2 q}\right)
\end{equation}
for a suitable constant $C$ and majorize summation over primes in $S_i$ by $y_i$-rough numbers
\[
S_i \ll \sum_{\substack{Y_i<n\le Y_{i+1} \\ P^{-}(n)\ge y_i}}\frac{1}{n}\left|\sum_{\chi\in \cE}\vartheta(\chi)\chi(n)\right|^2,
\]
where $P^{-}(n)$ denotes the smallest prime factor of an integer $n$. After expanding the square and rearranging, we get 
\begin{equation}
\label{eq:Si}
S_i\ll \sum_{\chi_1,\chi_2\in \cE}\sum_{\substack{Y_i<n\le Y_{i+1} \\ P^{-}(n)\ge y_i}}\frac{(\chi_1\overline{\chi}_2)(n)}{n}.
\end{equation}

Next, we appeal to a character sum estimate (taken with $t=0$) of 
Koukoulopolous~\cite[Lemma~2.4]{Kouk}, which in our setting implies 
that if $x\ge q^{4}$ then for a sufficiently large $y$,
\begin{equation}
\label{eq:kouk}
\sum_{\substack{n\le x \\ P^{-}(n)>y}}\chi(n)=\frac{\delta(\chi)\phi(q)}{q} x \prod_{\substack{p\le y\\ p\nmid q}}\left(1-\frac{1}{p}\right)+O\left(\frac{x^{1-1/(30\log{y})}}{\log y} \right), 
\end{equation}
where $\delta(\chi)=1$ if $\chi$ is the trivial character and $0$ otherwise. By~\eqref{eq:kouk}, partial summation and our choice of $Y_{i}$ and $y_i$ in~\eqref{eq:Yichoice} and~\eqref{eq:yichoice}, respectively, we have
\begin{align*}
\sum_{\substack{Y_i\le n<Y_{i+1}\\ P^{-}(n)>y_i}}\frac{\chi(n)}{n}&\ll \delta(\chi)\frac{\log Y_i}{\log{y_i}}+ \exp\left(-\frac{\log{Y_{i}}}{30 \log{y_i}}\right)\\ 
&\ll \delta(\chi)\log_2 q+(\log{q})^{-C/30}.
\end{align*}

Substituting this into~\eqref{eq:Si}, we obtain 
\[
S_i \ll (\#\cE)\cdot(\log_2 q)+(\#\cE)^2(\log{q})^{-C/30}
\]
and hence from~\eqref{eq:EESI}
\[
(\log{q})^{-2c}(\#\cE)^2\ll (\#\cE)\cdot (\log_2 q)^{3}+(\#\cE)^2 (\log{q})^{-C/31}
\]
which after choice of $C\geq 62 c$ implies that 
\[
\#\cE\ll (\log{q})^{2c} (\log_2 q)^{3}
\] 
and completes the proof.
\end{proof}

\section{Proofs of  bounds on  character sums and zero density results}  
\subsection{Proof of Theorem~\ref{thm:MeanVal}} 

By the Cauchy--Schwarz inequality
\begin{equation}\label{eq:M Cauchy} 
M(\balpha,\cA, N)^2 \le \frac{1}{A} \sum_{\chi \in \cA} \left| \sum_{n=1}^N \alpha_n \chi(n)\right|^2. 
\end{equation}

Applying Lemma~\ref{lem:G-R Cover}, we see that 
\[
\frac{A}{2}   \sum_{\chi \in \cA} \left| \sum_{n=1}^N \alpha_n \chi(n)\right|^2  \le \sum_{\chi_1, \chi_2 \in \cA}  \sum_{\eta \in \cU} \left| \sum_{n=1}^N  \alpha_n \chi_1(n) \overline\chi_2(n) \eta(n) \right|^2. 
\]

Next,  for each character $\eta \in \cU$ we treat the values of $\eta(n)$ as
complex weights and apply Lemma~\ref{lem:H-B} to the double sum over $\chi_1, \chi_2 \in \cA$. 
This implies the bound 
\begin{align*}
\frac{A}{2}   \sum_{\chi \in \cA} \left| \sum_{n=1}^N \alpha_n \chi(n)\right|^2  &\le \# \cU  \(A N^2 + A^2N + A^{5/4} N q^{1/2}\) q^{o(1)}\\
&  \le  K  \(A N^2 + A^2N + A^{5/4}N q^{1/2}\) q^{o(1)}.
\end{align*} 
Recalling~\eqref{eq:M Cauchy}, we conclude the proof.

\subsection{Proof of Theorem~\ref{thm:mean-value-sub}}

The proof combines the usual zero detection with an estimate derived from Theorem~\ref{thm:MeanVal}.  We first recall several constructions from the proof of~\cite[Theorem~12.1]{Mont}. Let $X,Y$ be defined by
\begin{equation}
\label{eq:XYchoice}
X=q^{\varepsilon} \mand Y= \(A + A^{1/4} q^{1/2}\)^{1/(3-2\sigma)}
\end{equation} 
for some small and fixed $\varepsilon > 0$. We further define \[
a_n = \sum_{\substack{d\mid n\\d \le X}}\mu(d), 
\]
where  $\mu(d)$ is the M{\"o}bius function, and 
\[
M_X(s,\chi)=\sum_{n\le X}\frac{\mu(n)\chi(n)}{n^{s}}.
\]  By~\cite[Equations (12.25) and (12.26)]{Mont}, each $\rho=\beta+i\gamma$ satisfying 
\begin{equation}
\label{eq:Lrhochi}
L(\rho,\chi)=0, \quad \beta \ge \sigma, \quad |\gamma|\le (\log{q})^{O(1)}
\end{equation}
satisfies at least one of the following two inequalities:
\begin{equation}
\label{eq:T1}
\left|\sum_{X<n\le Y^{2}}a_n\chi(n)n^{-\rho}e^{-n/Y}\right|\ge \frac{1}{6}
\end{equation} 
or 
\begin{equation}
\label{eq:T2}
\left|\int_{-C(\log{q})}^{C(\log{q})} L\left(\frac{1}{2}+i\gamma+iu,\chi\right)M_X\left(\frac{1}{2}+i\gamma+iu,\chi\right)Y^{1/2-\beta+iu} du\right|\ge \frac{1}{6}, 
\end{equation}
where $C$ is an absolute constant.  

With $\varepsilon$ as above, since 
\[
\sum_{n\ge Y^{1+\varepsilon}}\frac{|a_n|}{n^{\beta}}e^{-n/Y}=o(1), 
\]
we see that~\eqref{eq:T1} implies 
\begin{equation}
\label{eq:T11}
\left|\sum_{X<n\le Y^{1+\varepsilon}}a_n\chi(n)n^{-\rho}e^{-n/Y}\right|\ge \frac{1}{7}, 
\end{equation} 
provided that $q$ is large enough.

Considering~\eqref{eq:T2}, we apply the triangle inequality and use the trivial bound 
\[
M_X\left(\frac{1}{2}+i\gamma+iu,\chi\right)\ll X^{1/2}=q^{\varepsilon/2}
\]
to get 
\begin{equation}
\begin{split}
\label{eq:T22}
\int_{-C\log{q}}^{C\log{q}}\left|L\left(\frac{1}{2}+i\gamma+iu,\chi\right)\right| du&\gg Y^{\beta-1/2}q^{-\varepsilon/2}\\
& \gg Y^{\sigma-1/2}q^{-\varepsilon/2},
\end{split} 
\end{equation}
since by assumption $\beta \ge \sigma$.   

Let $\cR_1(\chi)$ denote the set of zeros counted by $N(\sigma,T;\chi)$ satisfying~\eqref{eq:T1} and let $\cR_2(\chi)$ denote the set of such zeros satisfying~\eqref{eq:T2} (note that  $\cR_1(\chi)$ and  $\cR_2(\chi)$ may intersect).
 We also define 
\begin{equation}
\label{eq:r1212}
R_1 = \sum_{\chi\in \cA} \# \cR_1(\chi) \mand R_2 = \sum_{\chi\in \cA} \# \cR_2(\chi) 
\end{equation}
so that
\begin{equation}
\label{eq:123}
\sum_{\chi\in \cA}N(\sigma,T; \chi)\le R_1 + R_2.
\end{equation}

Consider first $R_1$. Performing the dyadic pigeonhole principle to summation over $n$ in~\eqref{eq:T11}, there exists 
\begin{equation}
\label{eq:Xlb}
X\le N \le Y^{1+\varepsilon}
\end{equation}
 such that 
\begin{equation}
\label{eq:kkk}
 R_1\ll \log{q}  \sum_{\chi\in \cA}  \sum_{\rho\in \cR_1(\chi) }\left|\sum_{N/2 \le n \le N}a_n\chi (n)n^{-\rho} e^{-n/Y}\right|.
\end{equation}

We next apply H\"{o}lder's inequality to raise~\eqref{eq:kkk} to a suitable power in order to control the length of summation.

 Either $N \le Y^{2/3}$ or $Y^{2/3}< N \le Y^{1+\varepsilon}$ and in each case we can choose an integer $k$ such that 
\begin{equation}
\label{eq:123123123123123123123123123123123123123}
Y^{4/3}\le N^{k}\le Y^{2+\varepsilon}.
\end{equation}
Note that by~\eqref{eq:XYchoice} and~\eqref{eq:Xlb} we have 
$k=O(1)$.
By H\"{o}lder's inequality and~\eqref{eq:kkk}, we get 
\[
 R_1^k\ll  R_1^{k-1}(\log q)^k \sum_{\chi\in \cA}  \sum_{\rho\in \cR_1(\chi) }\left|\sum_{N/2 \le n \le N}a_n\chi(n)n^{-\rho} e^{-n/Y}\right|^k, 
\]
which implies
\begin{equation}
\label{eq:R1-step1}
R_1 \le q^{o(1)} \sum_{\chi\in \cA}  \sum_{\rho\in \cR_1(\chi) }\left|\sum_{(N/2)^k\le n\le N^{k}}A_n\chi(n)n^{-\rho} \right|
\end{equation}
for some complex numbers $A_n=n^{o(1)}$. 

In preparation to apply Theorem~\ref{thm:MeanVal}, we now remove the dependence on $\beta,\gamma$  in the above sum over zeros  $\rho=\beta+i\gamma$.  By partial summation and~\eqref{eq:Lrhochi},
\begin{align*}
& \left|\sum_{(N/2)^k\le n\le N^{k}}A_n\chi(n)n^{-\beta-it} \right| \\
&   \qquad \qquad \le \frac{q^{o(1)}}{N^{k\sigma}}\left|\sum_{(N/2)^k\le n\le N^{k}}A_n\chi(n)\right|\\
&   \qquad \qquad \qquad \qquad \qquad +\frac{q^{o(1)}}{N^{k(\sigma+1)}}\int_{0}^{(1-2^{-k})N^{k}}\left|\sum_{(N/2)^k\le n\le (N/2)^{k}+x}A_n\chi(n)\right|dx, 
\end{align*}
which substituted into~\eqref{eq:R1-step1} results in
\begin{align*}
R_1&\le \frac{q^{o(1)}}{N^{k\sigma}} \sum_{\chi\in \cA}  \sum_{\rho\in \cR_1(\chi) }\left|\sum_{(N/2)^k\le n\le N^{k}}A_n\chi(n)\right| \\
 & \qquad +\frac{q^{o(1)}}{N^{k(\sigma+1)}}\int_{0}^{(1-2^{-k})N^{k}}\sum_{\chi\in \cA}  \sum_{\rho\in \cR_1(\chi) }\left|\sum_{(N/2)^k\le n\le (N/2)^{k}+x}A_n\chi(n)\right|dx.
\end{align*}
Taking a maximum over $x$ in the above integration, we get 
\begin{align*}
R_1&\le \frac{q^{o(1)}}{N^{k\sigma}} \sum_{\chi\in \cA}  \sum_{\rho\in \cR_1(\chi) }\left|\sum_{(N/2)^k\le n\le (N/2)^{k}+x_0}A_n\chi(n)\right| \\
&\le  \frac{q^{o(1)}}{N^{k\sigma}} \sum_{\chi\in \cA}  \# \cR_1(\chi) \left|\sum_{(N/2)^k\le n\le (N/2)^{k}+x_0}A_n\chi(n)\right|
\end{align*}
for some $x_0\le  (1-2^{-k})N^k$. For   $T = (\log q)^{O(1)}$  and  any $\chi\in \cA$, discarding 
any restrictions on the real parts of $\beta \in [0,1]$ of zeros $ \rho=\beta+i\gamma$,  by~\eqref{eq: N(1,T)} we have 
$ \# \cR_1(\chi) \le N(T; \chi) \ll T \log T \le q^{o(1)}$. Hence, 
the above inequality implies  
\[
R_1\le  \frac{q^{o(1)}}{N^{k\sigma}}\sum_{\chi\in \cA}\left|\sum_{(N/2)^k\le n\le N^{k}}\widetilde{A}_n\chi(n) \right|, 
\] 
for some complex numbers $\widetilde{A}_n$ satisfying $|\widetilde{A}_n|\le n^{o(1)}$. We now apply Theorem~\ref{thm:MeanVal}, which  gives 
\begin{align*}
R_1& \le  \frac{q^{o(1)}}{N^{k\sigma}}K^{1/2}  \(AN^{k/2} + A^{1/2}N^{k}+ A^{5/8} N^{k/2} q^{1/4}\)\\
& = K^{1/2}  \(AN^{(1/2-\sigma)k} + A^{1/2} N^{(1-\sigma)k}+ A^{5/8}  N^{(1/2-\sigma)k} q^{1/4} \)q^{o(1)}.
\end{align*}
After using~\eqref{eq:123123123123123123123123123123123123123}, we get 
\begin{equation}
\label{eq:R1}
R_1\le K^{1/2} \left(Y^{2(1-\sigma)}A^{1/2}+\frac{A+  A^{5/8}  q^{1/4}}{Y^{4(\sigma -1/2)/3}}\right)  q^{o(1)}.
\end{equation}  

We next estimate $R_2$. Define $\cA^{\pm}$ by 
\[
\cA^{\pm }=\{ \chi\in \cA :~\chi(-1)=\pm 1\}.
\]

 Recall~\eqref{eq:r1212} and partition $R_2$ as
\[
 R_2 = \sum_{\substack{\chi\in \cA^{+}}} \# \cR_2(\chi)+ \sum_{\substack{\chi\in \cA^{-}}} \# \cR_2(\chi)=R_2^{+}+R_2^{-}.
\]

We consider only $R_2^{+}$, a similar argument and bound applies to $R_2^{-}$. By Lemma~\ref{lem:approx}, if  $t\le (\log{q})^{O(1)}$
then 
\[
L\left(\frac{1}{2}+it,\chi \right)\ll \left|\sum_{n\le q^{1/2+\varepsilon}}\frac{\alpha(n,1,t)\chi(n)}{n^{1/2}} \right|+1, 
\]
with  $|\alpha(n,1,t)|\ll 1$. 
Substituting into~\eqref{eq:T22}  results in 
\begin{align*}
 Y^{\sigma-1/2}R^{+}_2 
 & \ll q^{\varepsilon/2} \sum_{\chi \in \cA^{+}}\sum_{\rho\in \cR_2(\chi)}\int_{-C\log{q}}^{C\log{q}}\left(\left|\sum_{n\le q^{1/2+\varepsilon}}\frac{\alpha(n,1,\gamma+u)\chi(n)}{n^{1/2}} \right|+1\right)du \\ 
& \ll  q^{\varepsilon/2} \sum_{\chi \in \cA^{+}}\sum_{\rho\in \cR_2(\chi)}\int_{-C\log{q}}^{C \log{q}}\left|\sum_{n\le q^{1/2+\varepsilon}}\frac{\alpha(n,1,\gamma+u)\chi(n)}{n^{1/2}} \right|du +q^{\varepsilon}R^{+}_2.
\end{align*} 

Extending the summation to the whole set $\cA$, we write
\[
 Y^{\sigma-1/2}R^{+}_2   \ll  q^{\varepsilon/2} \sum_{\chi \in \cA}\sum_{\rho\in \cR_2(\chi)}\int_{-C\log{q}}^{C \log{q}}\left|\sum_{n\le q^{1/2+\varepsilon}}\frac{\alpha(n,1,\gamma+u)\chi(n)}{n^{1/2}} \right|du +q^{\varepsilon}R^{+}_2.
\]

Recalling  the restriction on $\gamma$ in~\eqref{eq:Lrhochi},  and increasing $C$ if necessary, we obtain 
\[
 Y^{\sigma-1/2}R^{+}_2 
\ll   q^{\varepsilon/2} \sum_{\chi \in \cA}\sum_{\rho\in \cR_2(\chi)} \int_{-C(\log{q})^{C}}^{C(\log{q})^{C}} \left|\sum_{n\le q^{1/2+\varepsilon}}\frac{\alpha(n,1,u)\chi(n)}{n^{1/2}} \right|du +q^{\varepsilon}R^{+}_2, 
\]
which, as in the previous case,  eliminates the dependence on $\gamma = \Im \rho$.

Since $1> \sigma>1/2$ and we can assume that 
\[A\ge q^{(2-2\sigma)/(2-\sigma)},
\]  
(as otherwise the bound~\eqref{eq: TrivBound} is stronger that the desired result), by our choice of $Y$ in~\eqref{eq:XYchoice} and taking a sufficiently 
small $\varepsilon> 0$, the above implies that 
\[
Y^{\sigma-1/2}R^{+}_2  \ll   q^{\varepsilon/2} \int_{-C(\log{q})^{C}}^{C(\log{q})^{C}}\sum_{\chi \in \cA}\#\cR_2(\chi)\left|\sum_{n\le q^{1/2+\varepsilon}}\frac{\alpha(n,1,u)\chi(n)}{n^{1/2}} \right|du .
\]
 Next, using~\eqref{eq: N(1,T)},   we
estimate $ \# \cR_2(\chi) \le N(T; \chi) \ll T \log T \le q^{o(1)}$. Hence, we now derive
\[
Y^{\sigma-1/2}R^{+}_2 
\ll q^{\varepsilon}\sum_{\chi \in \cA}\left|\sum_{n\le q^{1/2+\varepsilon}}\frac{\widetilde\alpha(n)\chi(n)}{n^{1/2}} \right|
\]
for some complex numbers $\widetilde\alpha(n)$ satisfying $|\widetilde\alpha(n)|\ll 1$.

 By the dyadic pigeonhole principle, there exists $1\le N \le q^{1/2+\varepsilon}$ such that 
\[
Y^{\sigma-1/2}R^{+}_2 \ll \frac{q^{2\varepsilon}}{N^{1/2}}\sum_{\chi \in \cA}\left|\sum_{N\le n \le 2N}\hat\alpha(n)\chi(n)\right|
\]
for some complex numbers $\hat\alpha(n)$ satisfying $|\hat\alpha(n)|\ll 1$.

By Theorem~\ref{thm:MeanVal}, 
\begin{align*}
Y^{\sigma-1/2}R^{+}_2 & \ll q^{2\varepsilon}  K^{1/2}\(A+A^{1/2}N^{1/2} + A^{5/8}q^{1/4}\)\\
& \le q^{2\varepsilon} K^{1/2}  \(A+A^{1/2}q^{1/4+\varepsilon/2}+ A^{5/8}q^{1/4}\), 
\end{align*} 
which, since $\varepsilon >0$ is arbitrary,  implies
\[
R^{+}_2\le K^{1/2}  \frac{A + A^{5/8}q^{1/4}}{Y^{\sigma-1/2}} q^{o(1)}.
\]
A similar argument shows that 
\[
R^{-}_2\le K^{1/2} \frac{A + A^{5/8}q^{1/4}}{Y^{\sigma-1/2}} q^{o(1)}.
\]
Combining the above with~\eqref{eq:123} and~\eqref{eq:R1}, we get 
\[
\sum_{\chi\in \cA}N(\sigma,T; \chi)\le K^{1/2} \left(Y^{2(1-\sigma)}A^{1/2}+\frac{A + A^{5/8}q^{1/4}}{Y^{\sigma-1/2}} \right) q^{o(1)}.
\]
Finally, the result follows by the choice of $Y$ in~\eqref{eq:XYchoice}.

\section{Proofs of results on extreme values}

\subsection{Function $\Psi$}  
In what follows, for a character $\chi$ we define
\begin{equation}
\label{eq:Psidef}
\Psi(\chi)=\frac{1}{A^2}\#\{ (\chi_1,\chi_2)\in \cA^2:~ \chi=\chi_1 \overline{\chi}_2 \}
\end{equation}
where $A=\#\cA$. Note that for each $\chi$
\[
0\le \Psi(\chi)\le 1.
\]

For $\lambda \in \Z$ define $\widehat{\Psi}$ by 
\[\widehat{\Psi}(\lambda)=\sum_{\chi \in \cX_q}\Psi(\chi)\chi(\lambda)
\]
so that
\begin{equation}
\label{eq:Psi-1}
\widehat{\Psi}(\lambda)=\frac{1}{A}\left|\sum_{\chi\in \cA}\chi(\lambda) \right|^2\ge 0
\end{equation}
which also implies 
\begin{equation}
\label{eq:AAP1}
\widehat{\Psi}(1)=A.
\end{equation}
Obviously, $\widehat{\Psi}(\lambda)$ depends only on the residue of $\lambda$ modulo $q$.

Note that if $\cA$ is a subgroup, then $\Psi(\chi)$ is just the indicator function of $\cA$. We use the above properties repeatedly in what follows.
\subsection{Proof of Theorem~\ref{th1}}\label{sec:th1}
\label{sec:12}

\subsubsection{Proof of the unconditional lower bound~(i)}  
We follow the argument of the proof of~\cite[Theorem~1.1]{AMMP},  which is based on the variant of the resonance method introduced in~\cite{AMM}. Let
\[
\cD=\cA/\cA
\]
and let $\Psi$ be as in~\eqref{eq:Psidef}. We have
\hrule
\[
\Psi(\chi)
=\frac1{A^2}
\sum_{\chi_1,\chi_2\in\mathcal A}
\mathbf 1_{\{\chi=\chi_1\overline{\chi_2}\}}.
\]
Since $\chi(1)=1$, expanding and interchanging the sums gives
\begin{align*}
\widehat\Psi(1)
&=\sum_{\chi\in\mathcal X_q}\Psi(\chi)
=\frac1{A^2}\sum_{\chi_1,\chi_2\in\mathcal A}
\sum_{\chi\in\mathcal X_q}
\mathbf 1_{\{\chi=\chi_1\overline{\chi_2}\}}\\
&=\frac1{A^2}\sum_{\chi_1,\chi_2\in\mathcal A}
\sum_{\chi\in\mathcal D}
\mathbf 1_{\{\chi=\chi_1\overline{\chi_2}\}}=\sum_{\chi\in\mathcal D}\Psi(\chi).
\end{align*}

In particular, by~\eqref{eq:AAP1},
\begin{equation}
\label{eq:Psi-mass}
\widehat\Psi(1)=\sum_{\chi\in \cD}\Psi(\chi)=A.
\end{equation}
To establish Part~(i), we fix some parameter  $\kappa$ with 
  \begin{equation} \label{eq:kappa def} 
0 < \kappa  < \frac{\delta}{(1+\delta) \log 4}
  \end{equation}  
 which implies that 
  \begin{equation} \label{eq:kappa ineq} 
  (1+\delta)(1-\kappa \log 4) > 1. 
  \end{equation}  
Then we set 
 \begin{equation} \label{eq:X} X = \kappa  (\log A)\cdot \log_2 A. 
 \end{equation}

 We define $L(1,\chi; X)$ and the coefficients $a_k$, $ k =1, 2, \ldots$, by the relation
\begin{equation}
\label{eq:LX expansion} 
L(1,\chi; X) = \prod_{p \leq X} \left(1 - \frac{\chi(p)}{p} \right)^{-1} = \sum_{k=1}^\infty a_k \chi(k).
\end{equation}
 For primes $p \leq X$, we set
\begin{equation}
\label{eq:qp} 
r_p = 1-\frac{p}{X},
\end{equation}
while for primes $p > X$ we set $r_p = 0$, and we extend this to $r_m$ for general $m$ in a completely multiplicative way.  For a character $\chi \in \cX_q$, we define the resonator function $R(\chi)$ as
\[
R(\chi) = \prod_{p\leq X} (1-r_p\chi(p))^{-1} = \sum_{m=1}^\infty r_m \chi(m).
\]

Similarly to~\eqref{eq:LX expansion}, we define  the coefficients $b_k$, $ k =1, 2, \ldots$, by the relation
\[
L(1,\chi; q^{4}) = \prod_{p \leq q^{4}} \left(1 - \frac{\chi(p)}{p} \right)^{-1} = \sum_{k=1}^\infty b_k \chi(k).
\]

For $q$ large enough, we have $q \geq X$. Hence, by definition
the coefficients $a_k$ and $b_k$ are non-negative real numbers 
such that $b_k \geq a_k$ for all $k\ge 1$.

We now fix some $\eta$ with $0 < \eta < 2$ and such that 
\begin{equation}
\label{eq:eta} 
 1+ 6 \eta < (1+\delta)(1-\kappa \log 4),
\end{equation}   
which is possible by the choice of $\kappa$ satisfying~\eqref{eq:kappa ineq}.

By Lemma~\ref{lem:L1approx}, we have
\begin{equation} \label{lapprox}
L(1,\chi) = L(1,\chi;q^{4}) \left(1 + O \left(\frac{1}{(\log q)^{\eta}} \right) \right),
\end{equation}
for all characters $\chi \bmod q$ except for the principal character and an exceptional set of characters $\cE$ which satisfies 
\begin{equation}
\label{eq:T21-Ebound}
\#\cE\ll (\log{q})^{3\eta} . 
\end{equation}
Set 
\[
S_1 = \sum_{\chi \in \cD} \Psi(\chi)L(1,\chi; q^{4}) |R(\chi)|^2 
\mand 
S_2 = \sum_{\chi \in \cD} \Psi(\chi)|R(\chi)|^2.
\]

Since $q>X$, every integer on which the resonator is supported is invertible modulo $q$. Expanding the sums defining $S_1$ and $S_2$ we obtain
\begin{align*}
S_1
&=\sum_{\substack{k\ge 1\\\gcd(k,q)=1}}b_k\sum_{m,n\ge 1}r_mr_n
\widehat\Psi(kmn^{-1}),\\
S_2
&=\sum_{m,n\ge 1}r_mr_n\widehat\Psi(mn^{-1}),
\end{align*}
where the inverses are taken modulo $q$. By~\eqref{eq:Psi-1}, all terms on the right hand sides are non-negative. Using $b_k\ge a_k$, restricting the first sum to $n=k\ell$, and using the complete multiplicativity of $r_n$, we have
\begin{align*}
S_1
&\ge \sum_{k=1}^\infty a_k\sum_{m,\ell\ge 1}r_mr_{k\ell}
\widehat\Psi(km(k\ell)^{-1})\\
&=\sum_{k=1}^\infty a_kr_k\sum_{m,\ell\ge 1}r_mr_\ell
\widehat\Psi(m\ell^{-1})
=S_2\sum_{k=1}^\infty a_kr_k.
\end{align*}
Consequently,
\[
\frac{S_1}{S_2}  \geq \sum_{k=1}^\infty a_k r_k  
 = \prod_{p \leq X} \left(1 - r_p p^{-1} \right)^{-1}.
\]
Using~\cite[Equations~(2.7) and~(2.8)]{AMMP} to estimate the above product with  $r_p$ given by~\eqref{eq:qp}, and then recalling~\eqref{eq:X}, we arrive at 
the following analogue of~\cite[Equation~(2.9)]{AMMP}
\begin{equation}  
\begin{split} \label{eq: s1s2ratio}
\frac{S_1}{S_2} &\geq e^{\gamma_0} \log X \left(1 - \frac{1}{\log X} +O\left(\frac{1}{(\log X)^2} \right) \right)  \\
& =   e^{\gamma_0} \left(\log_2 A + \log_3 A + \log \kappa - 1 + o\left(1\right) \right).
\end{split}
\end{equation}

We note that by the Prime Number Theorem 
\begin{equation}  \label{eq: boundr}
\log |R(\chi)|^2  \le \log |R(\chi_0)|^2   = \frac{2X}{\log X}+O\left(\frac{X}{(\log X)^2}\right)
\end{equation}  
and that by the Mertens theorem (see, for example,~\cite[Equation~(2.16)]{IwKow}), 
\begin{equation} \label{eq: L bound chi0}
 \left|L(1,\chi_*;q^{4})\right| \ll \log q , \qquad \chi_* \in  \cE\cup\{\chi_0\}.
\end{equation} 
  
We next give a lower bound for $S_2$. Retaining only the diagonal terms $m=n$ and using~\eqref{eq:Psi-mass}, we obtain
\begin{equation}  \label{eq: LB S2}
S_2\ge A\sum_{m=1}^\infty r_m^2.
\end{equation}
By Lemma~\ref{lem:q2},
\[
\sum_{m=1}^\infty r_m^2 \ge  \exp\left((2-\log{4})\frac{X}{\log X}+O\left(\frac{X  \log_2 X}{(\log X)^2}\right)\right).
\]
We now derive 
from~\eqref{eq: LB S2}  that 
\begin{equation}\label{lower_S2}
S_2 \geq \exp\left((2-\log{4})\frac{X}{\log X}+\log{A}+O\left(\frac{X \log_2 X}{(\log X)^2}\right)\right).
\end{equation}

We next give an upper bound for terms involving $\left|L(1,\chi_*; q^{4})\right|\left|R(\chi_*)\right|^2$. Combining~\eqref{eq: boundr} with~\eqref{eq: L bound chi0}, we have 
\begin{equation} \label{eq: LR bound chi0}
\left | L(1,\chi_*; q^{4})\right|\left|R(\chi_*)\right|^2   \leq \exp\(  \frac{2X}{\log X}+\log_2 q+O\left(\frac{X}{(\log X)^2}\right) \),
\end{equation} whenever $\chi_* \in  \cE\cup\{\chi_0\}$.

We next show that with our choice of parameters~\eqref{eq:X}  we have 
\begin{equation} \label{eq: LR and S2}
 \left|L(1,\chi_{*}; q^{4})\right|\left|R(\chi_*)\right|^2=O\left(\frac{S_2}{(\log{q})^{4\eta}}\right), \qquad \chi_* \in \cE\cup\{\chi_0\}.
\end{equation}
We see from~\eqref{lower_S2} and~\eqref{eq: LR bound chi0} that this is satisfied provided
\[
\frac{X \log 4}{\log X}\le \log{A}-(1+4\eta)\log_2{q}+ O\(\frac{X \log_2 X}{(\log X)^2}\)
\]
or 
\begin{equation} \label{eq: X vs H-q}
\( \log 4 + o(1)\)\frac{X}{\log X}\le \log{A}-(1+5\eta)\log_2{q}. 
\end{equation}
In turn, by our choice of $X$ in~\eqref{eq:X}, we have 
\[
\frac{X}{\log X}   = \(\kappa +o(1)\) \log A, 
\]
as $A\to \infty$.
Thus, 
 the inequality~\eqref{eq: X vs H-q} is satisfied for sufficiently large $q$ provided 
\[
\(1- \kappa \log 4\) \log{A}\geq  (1+6\eta)\log_2{q} 
\]
which, for   $A\geq  (\log q)^{1+\delta}$ is ensured by~\eqref{eq:eta}.
Hence we also have~\eqref{eq: LR and S2}. 

Put $\cE^*=\cE\cup\{\chi_0\}$ and set 
\[
S_1^*= \sum_{\chi \in \cD\setminus \cE^*} \Psi(\chi)L(1,\chi;  q^{4}) |R(\chi)|^2 \quad 
\text{and}\quad S_2^*  = \sum_{\chi \in \cD\setminus \cE^*} \Psi(\chi)|R(\chi)|^2.
\]

To get a lower bound on  the ratio $|S_1^*|/S_2^*$, we first notice that $0 \le S_2^* \le S_2$. Since $S_1$ is real and non-negative, from~\eqref{eq:T21-Ebound},~\eqref{eq: LR and S2} and $0\le\Psi\le1$, we derive 
\[
\frac{|S_1^*|}{S_2^*}  \geq \frac{|S_1^*|}{S_2}
\ge \frac{S_1}{S_2} + O\left(\frac{\#\cE^*}{(\log{q})^{4\eta}}\right)
=\frac{S_1}{S_2} + O\(\frac{1}{(\log q)^{\eta}}\).
\]  
Recalling~\eqref{eq: s1s2ratio},  we obtain 
\[
\frac{|S_1^*|}{S_2^*} 
 \geq  e^{\gamma_0} \(\log_2 A + \log_3 A + \log \kappa - 1   + o(1) \).
\]
It follows from~\eqref{lapprox} that there is a non-principal character $\chi \in \cD$ for which
\[
|L(1,\chi)|  \geq  e^{\gamma_0} \(\log_2 A + \log_3 A + \log \kappa - 1   + o(1) \).
\]
Since $\kappa$ is arbitrary with~\eqref{eq:kappa def}, this completes the proof of Part~(i).

\subsubsection{Proof of the conditional lower bound~(ii)}
We now assume the GRH for all Dirichlet $L$-functions of modulus $q$. Applying Lemma~\ref{approxDir} with $\sigma_0=1/2$ and $Y=(\log q)^2(\log_2 q)^6$, we have (see the discussion after~\cite[Lemma~2.1]{GrSo-2})  
\[
L(1,\chi) = L(1,\chi; Y) \left(1 + O \left(\frac{\log_3 q}{\log_2 q} \right) \right).
\]
for all non-principal characters $\chi \in \cX_q$.

We now argue exactly as in the first part of the proof, using the same weight $\Psi$, and also use 
that by  the Mertens formula (see~\cite[Equation~(2.16)]{IwKow}) we have the bound $L(1, \chi_0; Y)\ll \log_2 q$. 
The only change comes from the following 
estimate replacing~\eqref{eq: LR bound chi0}
\[
|L(1,\chi_0; Y)|\,|R(\chi_0)|^2 \leq \exp\( (1+o(1))\left(2\kappa +\frac{\log_3 q}{\log A}\right)  \log A \),
\]
with the same choice of $X=\kappa (\log A)\cdot \log_2 A$ and $\kappa$.

Thus, under the weaker condition $A \geq (\log_2 q)^{1+\delta}$, we can finish the proof 
following exactly the same steps as in the unconditional case above.

\subsection{Proof of Theorem~\ref{th3}}

\subsubsection{Proof of the unconditional lower bound~(i)}
We follow the ideas of~\cite{AMMP}. Let
\[
\cD=\cA/\cA \mand D=\#\cD,
\]
and let $\Psi$ be as in~\eqref{eq:Psidef}.  By the Pl{\"u}nnecke--Rusza inequality~\eqref{eq:Plunnecke-second-difference}, we have
\[
A\le D\le K^2A \mand \#\cD^2\le K^4 D.
\]
Let $N(\sigma,T, \chi)$ denote the number of zeros of $L(s,\chi)$ such that $\Re s\geq \sigma$ and $|\Im s|\leq T$. Applying Theorem~\ref{thm:mean-value-sub} to $\cD$ (and thus with $K^4$ instead of $K$), we obtain
 \begin{equation}
\label{eq:Hdensity}
\sum_{\chi\in \cD}N(\sigma,T; \chi)\ll  \(A^{(7-6\sigma)/(6-4\sigma)} +  A^{(4-3\sigma)/(6-4\sigma)} q^{(1-\sigma)/(3-2\sigma)}\) q^{o(1)}, 
\end{equation}
provided $T=(\log{q})^{O(1)}$.

Thus, combining~\eqref{eq:Hdensity} and Lemma~\ref{approxDir} with 
$\sigma_0=\sigma-\eta>1/2$, for some fixed but sufficiently small $\eta$, $X=(\log q)^{4/\eta}$, $t=0$ and $T=X+2$,  we obtain
\begin{equation}\label{approxsum}
 \log L(\sigma,\chi)=\sum_{2 \leq n \leq X} \frac{\Lambda(n)\chi(n)}{n^{\sigma}\log n}+O\left(\frac{1}{\log q}\right),
\end{equation} 
for all characters $\chi\in \cD$ except for a set of ``bad" characters $\sfB_\sigma(\cD)$ of cardinality 
\[
\#\sfB_\sigma(\cD)\ll  \(A^{(7-6\sigma_0)/(6-4\sigma_0)} +  A^{(4-3\sigma_0)/(6-4\sigma_0)} q^{(1-\sigma_0)/(3-2\sigma_0)}\) q^{o(1)}. 
\]
 If $\eta>0$ is small enough, simple calculus shows that for $A   \geq q^{(2-2\sigma)/(2-\sigma) + \delta}$ we have 
\begin{equation} 
\label{eq:Hdensity1}
\#\sfB_\sigma(\cD) \ll A^{1- \kappa }
\end{equation}
for some $\kappa> 0$ which depends only on $\delta$, $\eta$ and $\sigma$.

 We set $\sfB_\sigma^*(\cD)= \sfB_\sigma(\cD) \cup \{\chi_0\}$, where $\chi_0$ is the trivial character modulo $q$
 and consider the set of ``good" characters $\sfG_{\sigma}(\cD) = \cD \setminus \sfB_\sigma^*(\cD)$.
 Thus, we have
\[
\log L(\sigma,\chi) =  \sum_{2 \leq n \leq X} \frac{\Lambda(n)\chi(n)}{n^{\sigma}\log n}
+ O\left(\frac{1}{(\log q)}\right), \qquad \chi \in  \sfG_\sigma(\cD).
\]   

It follows that, for a fixed $\sigma>1/2$, 
\begin{equation} \label{S:approx}
\log L(\sigma,\chi) = S_{\chi}(\sigma,X) + O(1), \qquad \chi \in \sfG_\sigma(\cD),
\end{equation}
where 
\[
S_{\chi}(\sigma,X)= \sum_{p \leq X} \frac{\chi(p)}{p^{\sigma}}.
\]
Consequently, in order to prove Theorem~\ref{th3}, it suffices to exhibit a large value of $\Re S_{\chi}(\sigma,X)$ for $\chi \in \sfG_\sigma(\cD)$. 

Let 
\begin{equation} \label{eq: Y def}
Y=0.5 \kappa (\log A) \log_2 A, 
\end{equation}  
where $\kappa$ is as in~\eqref{eq:Hdensity1}. 

 Furthermore, we set $r_1=1$ and $r_p=0$ for $p>Y$, and $r_p=\tfrac12$ for all primes $p\leq Y$. We extend the definition of $r_p$ in a completely multiplicative way to obtain weights $r_n$ for all $n\geq 1$. We now define for $\chi\in \cX_q$ 
\begin{equation} \label{res:def}
R(\chi)=\prod_{p\leq Y} \left(1-r_p\chi(p)\right)^{-1} = \prod_{p\leq Y} \left(1-\frac{\chi(p)}{2} \right)^{-1}= \sum_{n=1}^\infty r_n \chi(n).
\end{equation}
For $q$ sufficiently large, we have $Y<X<q$. Since $q$ is prime, every integer on which the resonator is supported and every prime in $S_\chi(\sigma,X)$ is invertible modulo $q$.

Similar to the previous section, we consider the sums
\[
S_1^{\sfG} = \sum_{\chi \in  \sfG_\sigma(\cD)} \Psi(\chi)S_{\chi}(\sigma,X)\left|R(\chi)\right|^2 
\mand 
S_2^{\sfG} =\sum_{\chi \in  \sfG_\sigma(\cD)} \Psi(\chi)\left|R(\chi)\right|^2,
\]
as well as the sums 
\[
S_1 = \sum_{\chi \in \cD} \Psi(\chi)S_{\chi}(\sigma,X)\left|R(\chi)\right|^2 
\mand 
S_2=\sum_{\chi \in \cD} \Psi(\chi)\left|R(\chi)\right|^2.
\]
Since $S_1$ is real and non-negative, we use the inequality  
\begin{equation}\label{eq: max S}
\max_{\chi \in \sfG_\sigma(\cD)} \Re S_{\chi}(\sigma,X) 
\geq \frac{ \Re S_1^{\sfG}}{S_2^{\sfG} }  
\geq \frac{ \Re S_1^{\sfG}}{S_2}, 
\end{equation} 
provided the final numerator is positive.

Expanding $S_2$ and using~\eqref{eq:AAP1}, we have
\begin{equation}\label{S2tout}
S_2=\sum_{m,n=1}^\infty r_mr_n\widehat\Psi(mn^{-1})
\ge \widehat\Psi(1)\sum_{m=1}^\infty r_m^2\ge A\sum_{m=1}^\infty r_m^2.
\end{equation}  

Next, from~\eqref{res:def} and our choice of $Y$, we immediately obtain
\[
|R(\chi)|^2 \leq 2^{2 \pi(Y)} \leq \exp((2\log 2)(1+o(1))Y/\log Y) \leq A^{\kappa \log 2 + o(1)}, 
\]
where $\pi(Y)$ denotes the number of primes $p\le Y$. 
Hence, using 
\[
|S_{\chi}(\sigma,X)| \leq \sum_{p=2}^{X}\frac{1}{p^{\sigma}} \ll X^{1-\sigma} = (\log q)^{4\frac{(1-\sigma)}{\eta}}=q^{o(1)}, 
\]
we derive, recalling~\eqref{eq:Hdensity1},
\begin{equation}\label{S1bad}
 \left|\sum_{\chi \in \sfB_{\sigma}^*(\cD)} \Psi(\chi)S_{\chi}(\sigma,X) |R(\chi)|^2\right| 
 \le A^{1-\kappa + \kappa \log 2} q^{o(1)} 
 \ll A^{1-0.3\kappa}.  
\end{equation} 

On the other hand, expanding $|R(\chi)|^2$ and switching the order of summation, we have 
\begin{align*}
S_1
&=\sum_{p\le X}\frac{1}{p^\sigma}\sum_{m,n\ge1}r_mr_n
\widehat\Psi(pmn^{-1}).
\end{align*}
The inner terms are non-negative by~\eqref{eq:Psi-1}. Restricting to $n=p\ell$ and using the complete multiplicativity of $r_n$, we obtain
\[
\sum_{m,n\ge1}r_mr_n\widehat\Psi(pmn^{-1})
\ge r_p\sum_{m,\ell\ge1}r_mr_\ell\widehat\Psi(m\ell^{-1})
=r_pS_2.
\]
Recalling that $r_p = 0.5$ for $p \leq Y$ and $r_p=0$ for $p>Y$, we derive
\begin{equation}\label{eq: S1/S2} 
 \frac{S_1}{S_2}  \geq  \sum_{p=2}^{Y}\frac{1}{2p^{\sigma}}  
 \geq \frac{Y^{1-\sigma}(1+o(1))}{2(1-\sigma)\log Y}.
\end{equation}
 
We also deduce from~\eqref{S2tout} and~\eqref{S1bad} that, for fixed $\sigma$,
\[
\frac{\Re S_1^{\sfG}}{S_2}
\ge \frac{S_1}{S_2}-O\(A^{-0.3\kappa}\).
\]
Together with~\eqref{eq: max S},~\eqref{eq: S1/S2} and the choice of $Y$ in~\eqref{eq: Y def}, this implies 
\[
 \max_{\chi \in \sfG_\sigma(\cD)} \Re S_{\chi}(\sigma,X) \gg(\log A)^{1-\sigma} (\log_2 A)^{-\sigma}.
\]
On the set $ \sfG_\sigma(\cD)$, the sum $S_{\chi}(\sigma,X)$ is a good approximation to $\log L(\sigma,\chi)$, as noted in~\eqref{S:approx}. Taking real parts completes the proof of Part~(i).

\subsubsection{Proof of the conditional lower bound~(ii)} 
Assuming the GRH, we apply~\eqref{approxsum} with $t = 0$, $y = (\log q)^{4/(\sigma-1/2)}$ and $\sigma_0 = 1/2$. We obtain for $\chi\neq \chi_0$,
\[
\log L(\sigma,\chi)= \sum_{n=2}^{y}\frac{\Lambda(n)\chi(n)}{n^{\sigma}\log n} +O\left(\frac{1}{(\log q)^{2}}\right).
\]
As before, this gives
\[
\log L(\sigma,\chi)=S_\chi(\sigma,y)+O(1), \qquad \chi\neq\chi_0.
\]
Setting 
\begin{equation}
\label{eq:Ychoice1/2}
Y= \frac1B (\log A) \log_2 A, 
\end{equation}
we define the coefficients $r_n$ and the resonator $R(\chi)$ by~\eqref{res:def} with this value of $Y$. We then derive
\[
|R{(\chi_0)}|^2\leq \exp\left(\frac{2\log 2}{B}(1+o(1))\log A\right),
\]
and 
\[
|S_{\chi_0}(\sigma, y)||R{(\chi_0)}|^2\leq\exp\left((1+o(1))\left(\frac{2\log 2}{B}\log{A}+b(\sigma)\log_2 q\right)\right), 
\]
where $b(\sigma)=8(1-\sigma)/(2\sigma-1)$.

Using properties of $\Psi$ and retaining the diagonal terms, we obtain
\begin{equation}
\label{eq: sum R2}
S_2=\sum_{\chi\in \cD}\Psi(\chi)|R(\chi)|^2
\ge A\sum_{m=1}^{\infty}r_m^2.
\end{equation}
Recalling the definition of the coefficients $r_m$, in particular their complete multiplicativity, similarly to~\eqref{res:def}, we see that 
\[
\sum_{m=1}^{\infty}r_m^2 =  \prod_{p\le Y}\left(1-r_p^2\right)^{-1} =  \prod_{p\le Y}\left(1-\frac{1}{4}\right)^{-1}
=\( \frac{4}{3}\)^{\pi(Y)}.
\] 
From~\eqref{eq:Ychoice1/2},~\eqref{eq: sum R2} and the Prime Number Theorem, this implies 
\[
S_2\geq A \exp\left(\frac{\log (4/3)}{B}(1+o(1))\log A\right).
\]

The positivity argument used in the unconditional case gives
\[
S_1:=\sum_{\chi\in\cD}\Psi(\chi)S_\chi(\sigma,y)|R(\chi)|^2
\ge S_2\sum_{p\le Y}\frac{1}{2p^\sigma}.
\]
Removing the principal character, and using $0\le\Psi(\chi_0)\le1$, we deduce that 
\begin{align*}
\max_{\substack{\chi\in \cD\\ \chi\neq \chi_0}}\Re S_\chi(\sigma, y)
& \geq (1+o(1))  \frac{Y^{1-\sigma}}{2(1-\sigma)\log Y} \\ 
&\quad -O\left(\exp\left(\left(\frac{\log{3}}{B}-1+o(1)\right)\log A\right.\right.\\
&\hspace{45mm}\left.\left.+\left(b(\sigma)+o(1)\right)\log_2 q\right)\right).
\end{align*}
Choosing  
\[
B =\frac{2(1+ \delta)\log 3}{\delta},
\]
the assumption that 
\[A\ge (\log{q})^{(1+\delta)b(\sigma)}\]
implies 
\begin{align*}
\frac{\log{3}}{B}\log{A}+b(\sigma)\log_2 q-\log{A}
&\le\frac{\delta}{2(1+\delta)}\log {A}+\frac{1}{1+\delta}\log {A}-\log{A}\\
&=-\frac{\delta}{2(1+\delta)}\log A.
\end{align*}
Hence, recalling the choice of $Y$ in~\eqref{eq:Ychoice1/2}, we conclude that
\begin{align*}
\max_{\substack{\chi\in \cD\\ \chi\neq \chi_0}}\Re S_\chi(\sigma, y)
& \geq (1+o(1)) \frac{Y^{1-\sigma}}{2(1-\sigma)\log Y}\\
&=\frac{(1+o(1))}{2(1-\sigma)}\left(\frac{\delta}{2(1+\delta) \log{3}}\right)^{1-\sigma}\frac{(\log{A})^{1-\sigma}}{(\log_2{A})^{\sigma}}.
\end{align*}
Taking real parts in the approximation to $\log L(\sigma,\chi)$ completes the proof.

\subsection{Proof of Theorem~\ref{th4}}
We follow the proof of~\cite[Th{\'e}or{\`e}me~1.5]{dlbTen} with the necessary changes. Define
\[
\cA^{\pm}=\{\chi\in\cA:~\chi(-1)=\pm1\}
\]
and choose one of these sets, denoted by $\cA_0$, with cardinality $A_0=\#\cA_0\ge A/2$. 
Let $\Psi$ be as in~\eqref{eq:Psidef}, applied to $\cA_0$, and put
\[
\cD_0=\cA_0/\cA_0\subseteq \cA/\cA.  
\] 
Every character in $\cD_0$ is even. Since $q$ is prime, every non-principal character in $\cD_0$ is also primitive. Moreover,
\begin{equation}
\label{eq:Psi-even-mass}
\widehat\Psi(1)=\sum_{\chi\in\cD_0}\Psi(\chi)=A_0\ge A/2.
\end{equation}
If $\chi\in \cD_0$ is non-principal then from~\cite[Lemma~2]{Sound}, we have
\begin{equation}\label{approximation}
\left|L\(\tfrac{1}{2}, \chi\)\right|^2=2\sum_{k, \ell\geq 1}\frac{\chi(k)\overline{\chi(\ell)}}{\sqrt{kl}}W_0\left(\frac{\pi k\ell}{q}\right),
\end{equation}
where 
\[
W_0(x)=\frac{1}{2\pi i}\int_{\Re s = 2}\frac{\Gamma(1/4+s/2)^2}{\Gamma(1/4)^2 s}x^{-s}ds.
\]
By~\cite[Lemme~6.1]{dlbTen}, the weight function $W_0(x)$ is real-valued, smooth and positive on $(0, +\infty)$, bounded as $x$ approaches $0$ and decays rapidly as $x\to +\infty$. More precisely, $W_0(x)$ satisfies $W_0(x)= 1 + O(x^{\frac{1}{2}-\varepsilon})$ for small $x$, and $W_0(x)\ll 1/(1+x^2)$ for large $x$.  
Consider the following sums: 
\[
S_{1}=\sum_{\substack{\chi\in \cD_0\\ \chi\neq \chi_0}}\Psi(\chi)|R(\chi)|^2 
\mand 
S_{2}=\sum_{\substack{\chi\in \cD_0\\ \chi\neq \chi_0}}\Psi(\chi)\left|L\(\tfrac{1}{2}, \chi\)\right|^2|R(\chi)|^2,
\]
where $R(\chi)$ is a Dirichlet polynomial defined as follows. Let the set $\mathcal{M}$ be constructed following~\cite[Section~2.2]{dlbTen} provided that $\# \mathcal{M}\leq \hh$, where $\hh < q$ is to be optimised later.

Given a set $\widetilde{\cM}$ of representatives of the classes of $\mathcal{M}$ modulo $q$, for $m\in \widetilde{\cM}$, define
\[
r(m)^2=\sum_{\substack{n\in \mathcal{M}\\ n\equiv m \bmod{q}}}1.
\]
For $\chi\in \cX_q$, the corresponding resonator is defined as 
\[
R(\chi)=\sum_{m\in \widetilde{\cM}}r(m)\chi(m).
\]
Applying the Cauchy--Schwarz inequality, we have 
\begin{equation}\label{boundR}
|R{(\chi_0)}|^2\leq \# \widetilde{\cM} \sum_{m\in \widetilde{\cM}}r(m)^2 \leq \min\{q-1, \hh\}\hh \ll \hh^2,
\end{equation}
since $\# \widetilde{\cM}\leq \min\{q-1, \hh\}$ and from the definition of coefficients of resonator,
\begin{equation}\label{average over local frequescies}
\sum_{m\in \widetilde{\cM}}r(m)^2=\sum_{m\in \widetilde{\cM}}\sum_{\substack{n\in \mathcal{M}\\ n\equiv m \bmod{q}}}1=\sum_{n\in \mathcal{M}}1=\# \mathcal{M} \leq \hh.
\end{equation}
Notice that
\[
\max_{\substack{\chi\in \cD_0\\ \chi\neq \chi_0}}\left|L\(\tfrac{1}{2}, \chi\)\right|^2\geq \frac{S_{2}}{S_{1}}.
\] 
Utilising $0\le\Psi\le1$, the orthogonality relations for the full group and~\eqref{average over local frequescies}, we obtain the following bound for $S_{1}$:
\begin{equation}
\begin{split}
\label{eq:S1ub--1}
  S_{1}&\leq \sum_{\chi\in \mathcal{X}_q}|R(\chi)|^2\\
  & =\sum_{s, n \in \widetilde{\cM}}r(s)r(n) \sum_{\substack{\chi\Mod{q}}}\chi(s)\overline{\chi(n)}=(q-1) \sum_{\substack{s, n \in \widetilde{\cM}\\s\equiv n \Mod{q}}}r(s)r(n)\\
 & =(q-1)\sum_{n\in \widetilde{\cM}}r(n)^2\leq (q-1)\hh,
\end{split}
\end{equation}
where only $s=n$ survives as $q$ is a prime.

Using~\eqref{approximation}, expanding the resonator and adding back the principal character, we obtain
\begin{align*} 
 S_{2}
&=2\sum_{\substack{k, \ell\geq 1\\ \gcd(k\ell, q)=1}}\frac{1}{\sqrt{k\ell}}W_0\left(\frac{\pi k\ell}{q}\right)
\sum_{s,n\in\widetilde{\cM}}r(s)r(n)
\widehat\Psi(sk(n\ell)^{-1})\\
&\qquad +O\left(|R{(\chi_0)}|^2 q^{1/2}\log q\right).
\end{align*}
Here and below, inverses are taken modulo $q$. By~\eqref{eq:Psi-1}, all terms in the first expression are non-negative. Retaining the terms satisfying $sk\equiv n\ell\bmod q$, using~\eqref{eq:Psi-even-mass}, and then the positivity properties of $W_0$, we get
\begin{align*}
S_2
&\gg A\sum_{\substack{k, \ell\leq \sqrt{q}\\ \gcd(k\ell, q)=1}}\frac{1}{\sqrt{k\ell}}
\sum_{\substack{s, n \in \widetilde{\cM}\\sk\equiv n\ell \Mod{q}}}r(s)r(n)
-O\left(|R{(\chi_0)}|^2 q^{1/2}\log q\right).
\end{align*}
The decay of $W_0$ also implies that
\[
\sum_{\substack{k, \ell\geq 1\\ \gcd(k\ell, q)=1}}\frac{1}{\sqrt{k\ell}}W_0\left(\frac{\pi k\ell}{q}\right)\ll q^{1/2}\log q.
\]

From the relation $sk\equiv n\ell \Mod{q}$, observe that
\[
r(s)r(n)\geq \min\{r(s)^2, r(n)^2\}\geq \sum_{\substack{i, j\in \mathcal{M}\\ik=j\ell\\ i\equiv s\bmod q\\ j\equiv n \bmod q}}1.
\]
Using the above estimates and~\eqref{boundR}, we obtain 
\begin{equation}
\begin{split}
\label{eq:S2-gcd-sum}
 S_{2}&\gg A \sum_{\substack{k, \ell\leq \sqrt{q}\\ \gcd (k\ell, q)=1}}\frac{1}{\sqrt{k\ell}}\sum_{\substack{i, j\in \mathcal{M}\\ik=j\ell}}1-O\left(\hh^2 q^{1/2} \log q\right)\\ 
&\gg A \sum_{\substack{i, j\in \mathcal{M}\\i/\gcd(i, j), \, j/\gcd(i, j)\leq q^{1/2}\\ \gcd(ij, q)=1}}
 \(\frac{\gcd(i, j)}{\lcm[i, j]}\)^{1/2}-O\left( \hh^2 q^{1/2} \log q\right). 
\end{split}
\end{equation}
Next, we observe that for $i, j\in \mathcal{M}$ outside of the summation range in the last sum
in~\eqref{eq:S2-gcd-sum} we have $\lcm[i,j]/\gcd(i,j)>q^{1/2}$ and hence  
\[
 \(\frac{\gcd(i, j)}{\lcm[i, j]}\)^{1/2} < \frac{1}{q^{1/12}} \(\frac{\gcd(i, j)}{\lcm[i, j]}\)^{1/3}.
\]
Also note that if $\hh<q$ then the coprimality condition can be removed since the largest prime factor $y_\cM$ of elements of $\mathcal{M}$ for some fixed $\gamma<1$ can be bounded as 
\begin{equation}
\label{eq:yM}
y_\cM \ll \exp\((\log_2 h)^\gamma\) ( \log h) \log_2 h \le(\log h)^{1+o(1)}; 
\end{equation}
see the construction of $\cM$ in~\cite[Section~2.2]{dlbTen} and in particular~\cite[Equations~(2.1)--(2.3)]{dlbTen}.
Therefore
\begin{equation}
\begin{split}
\label{eq:Split}
& \sum_{\substack{i, j\in \mathcal{M}\\i/\gcd(i, j), \, j/\gcd(i, j)\leq q^{1/2}\\ \gcd(ij, q)=1}} \(\frac{\gcd(i, j)}{\lcm[i, j]}\)^{1/2}\\
& \qquad \qquad\quad =  \sum_{\substack{i, j\in \mathcal{M}\\i/\gcd(i, j), \, j/\gcd(i, j)\leq q^{1/2}}}\ \(\frac{\gcd(i, j)}{\lcm[i, j]}\)^{1/2}\\
&\qquad \qquad \quad \ge  \sum_{\substack{i, j\in \mathcal{M}}} \(\frac{\gcd(i, j)}{\lcm[i, j]}\)^{1/2}
-\frac{1}{q^{1/{12}}}
\sum_{\substack{i, j\in \mathcal{M}}}\(\frac{\gcd(i, j)}{\lcm[i, j]}\)^{1/3}  .
\end{split}
\end{equation}  

It is shown in~\cite[Page~128]{dlbTen} (in particular, see the  bound on $S_{1/3}(\cM)$ with 
$y_\cM$ as in~\eqref{eq:yM}) that 
\begin{equation}
\label{eq:S 1/3}
\sum_{\substack{i, j\in \mathcal{M}}}\(\frac{\gcd(i, j)}{\lcm[i, j]}\)^{1/3} \ll h\exp\(y_\cM^{2/3}\) 
\ll h\exp((\log h)^{2/3+o(1)}),
\end{equation}  
while by~\cite[Th\'{e}or\`{e}me~1.1]{dlbTen} to the first sum, we conclude that
\begin{equation}
\label{eq:S 1/2}
  \sum_{\substack{i, j\in \mathcal{M}}} \(\frac{\gcd(i, j)}{\lcm[i, j]}\)^{1/2}
    \ge \hh\exp\left((2\sqrt{2}+o(1))\sqrt{\frac{(\log \hh) \log_3 \hh}{\log_2 \hh}}\right).
\end{equation}  
Since $h < q$,  we see that the upper bound~\eqref{eq:S 1/3} is negligible compared 
to  the lower bound~\eqref{eq:S 1/2} and thus we conclude from~\eqref{eq:Split} that 
\[
 \sum_{\substack{i, j\in \mathcal{M}\\i/\gcd(i, j), \, j/\gcd(i, j)\leq q^{1/2}\\ \gcd(ij, q)=1}}\sqrt{\frac{\gcd(i, j)}{\lcm[i, j]}} 
 \ge \hh\exp\left((2\sqrt{2}+o(1))\sqrt{\frac{(\log \hh) \log_3 \hh}{\log_2 \hh}}\right). 
\]
After substitution in~\eqref{eq:S2-gcd-sum}, this yields 
\begin{equation}
\label{eq:S2-123123123123123}
S_{2}\gg A \hh\exp\left((2\sqrt{2}+o(1))\sqrt{\frac{(\log \hh) \log_3 \hh}{\log_2 \hh}}\right)-O\left(
\hh^2 q^{1/2}  \log q\right). 
\end{equation}

Recall that $q > A\ge q^{1/2 + \varepsilon}$. 
We choose $\hh=\lfloor A/\sqrt{q}\rfloor$. With this choice, we have 
\[
A \hh\exp\left((2\sqrt{2}+o(1))\sqrt{\frac{(\log \hh) \log_3 \hh}{\log_2 \hh}}\right)\gg  \frac{A^2}{q^{1/2}}\exp\left(c(\varepsilon) \sqrt{\frac{\log q \log_3 q}{\log_2 q}}\right)
\]
for an absolute constant $c(\varepsilon) >0$, depending only on $\varepsilon$, which dominates 
\[
\hh^2q^{1/2}\log{q} \ll \frac{A^2(\log{q})}{q^{1/2}}. 
\]
Hence by~\eqref{eq:S2-123123123123123}
and~\eqref{eq:S1ub--1} 
we conclude that 
\[
\max_{\substack{\chi\in \cD_0\\ \chi\neq \chi_0}}\left|L\(\tfrac{1}{2}, \chi\)\right|^2\geq \frac{S_{2}}{S_{1}}\gg \frac{A}{q}\exp\left((2\sqrt{2}+o(1))\sqrt{\frac{(\log \hh) \log_3 \hh}{\log_2 \hh}}\right).
\]  
Since for $q \ge A\ge q^{1/2 + \varepsilon}$ we have
$\log \hh=\log(A/\sqrt q)+o(1)$, $\log_2 \hh = (1+o(1))\log_2 A$ and $\log_3 \hh=(1+o(1))\log_3 A$, and since $\cD_0\subseteq
 \cA/\cA$, this completes the proof.

\section{Distribution of small powers of primitive roots}   
\label{sec:Appl} 

We now present some key consequences of Theorem~\ref{thm:MeanVal}.

Let $g$ be a fixed primitive root modulo $q$. 
Given integers $a$, $H$ and $N$, we denote by $f(a,H, N)$ the number 
of integers $n\in [1, N]$  such that $g^n \equiv a + h \bmod q$ for some integer $h\in [1,H]$.

Montgomery~\cite[Theorem~2]{Mont2} has introduced and estimated the variance
\[
V(H,N) =\frac{1}{q} \sum_{a=0}^{q-1} \(f(a,H, N)-\frac{HN}{q}\)^2
\]
of this quantity. In   the extreme regime when $H = O(q/N)$,  the expected value of 
$f(a,H, N)$ is $HN/q = O(1)$, and it is perhaps very difficult to say anything 
meaningful about  $V(H,N)$. However, 
by~\cite[Theorem~2]{Mont2} we have  $V(H,N) = (1 + o(1)) HN$ as $q\to \infty$ 
provided that
$HN \ge  q^{1+ \varepsilon}$ and $q \ge N\ge q^{5/7 + \varepsilon}$ for some fixed $\varepsilon > 0$. 
The result has been improved on average over $g$; see~\cite[Theorem~1]{CoGoZa}
and~\cite[Theorem~15.6]{KoSh} for various regimes of $H$ and $N$. 
Cobeli,  Gonek and Zaharescu~\cite[Theorem~2]{CoGoZa} have also shown that 
the range $N\ge q^{19/27 + \varepsilon}$ is admissible for almost all primes $q$.

One new range, which applies to all primes, improves both~\cite[Theorem~2]{Mont2} 
and  its ``almost all'' version~\cite[Theorem~2]{CoGoZa}. 

\begin{cor}
	\label{cor:MeanVal}
	Assume that  $HN   \ge  q^{1+ \varepsilon}$  and  $q> N\ge q^{2/3 + \varepsilon}$ with some fixed $\varepsilon > 0$. 
	Then  $V(H,N) = (1 + o(1)) HN$ as $q\to \infty$. 
\end{cor}

Next,  we consider the distribution of differences between elements 
of the sequence $g^n \bmod q$, $n =1, \ldots, N$, 
which has also been studied in a number of works; see, for example,~\cite{RuZa, ShpYau, VaZa}. 

Namely, we fix an interval $\cJ = [\alpha, \alpha + \gamma] \in [0,1]$ of length $\gamma$ and as in~\cite{RuZa} we  denote by $R_2(\cJ,N)$ the pair correlation 
function. More precisely, 
\begin{align*} 
R_2(\cJ,N) = \frac{1}{N} \#\bigl\{(m,n) :~1 \le m& \ne n \le N\\
&   g^m-g^n  \equiv  h \bmod q, \ h \in \cH\bigr\}, 
\end{align*} 
where  
\[
\cH = \frac{q}{H} \cdot \cJ =  \left[\alpha \frac{q}{H} , \(\alpha + \gamma\) \frac{q}{H} \right] .
\]
We say that the sequence $g^n \bmod q$, $n =1, \ldots, N$, has a Poissonian distribution 
of spacings if uniformly over $\alpha$ and $\gamma$ we have 
\begin{equation}\label{eq:PoissonCorr}  
R_2(\cJ,N)  = \gamma + o(1), \qquad  \text{as} \ q \to \infty.
\end{equation}  

It is easy to see that inserting the bound of Theorem~\ref{thm:MeanVal} in the argument 
of the proof of~\cite[Theorem~2]{Mont}, we derive the following improvement unconditionally, 
which in~\cite{Mont} is given under the GRH. 

We remark, that the approach of Montgomery~\cite{Mont} gives upper bounds on $V(H,N)$ in 
other regimes of $H$ and $N$; see~\cite[Equation~(15.3)]{KoSh}. Certainly using 
Theorem~\ref{thm:MeanVal} improves these bounds as well, which in turn have several other 
applications; we refer to~\cite[Chapter~15]{KoSh}.  

Furthermore, using the bound~\eqref{eq:MontBound}, Rudnick and Zaharescu~\cite[Theorem~3(i)]{RuZa}, have established~\eqref{eq:PoissonCorr}  
$N\ge q^{5/7 + \varepsilon}$ for some fixed $\varepsilon > 0$. 
In turn, Theorem~\ref{thm:MeanVal} immediately implies the following. 

\begin{cor}
	\label{cor:PoissonCor}
	For  $N\ge q^{2/3 + \varepsilon}$ with some fixed $\varepsilon > 0$ we have~\eqref{eq:PoissonCorr}. 
\end{cor}

We note that Corollary~\ref{cor:PoissonCor} achieves the same strength as the conditional 
on the GRH result of~\cite[Theorem~3(ii)]{RuZa} and the result using additional averaging 
over the primitive root $g$ of~\cite[Theorem~4]{RuZa}.

\section*{Acknowledgement}

During the preparation of this work, P.D., B.K.  and I.S. 
were supported by the  
Australian Research Council Grant  DP230100534, 
and M.M.  by AAP Recherche 2025 UJM ``Comportements al{\'e}atoires en arithm{\'e}tique'' 
and by the French National Research Agency (ANR) under project number ANR-25-CE40-1961-01.


\begin{thebibliography}{10}


\bibitem{AMM}  C. Aistleitner, K.  Mahatab and M. Munsch
`Extreme values of the {R}iemann zeta function on the $1$-line', {\it Int. Math. Res. Not.\/}, {\bf 22} (2019), 6924--6932. 


\bibitem{AMMP}  C. Aistleitner, K.  Mahatab, M. Munsch and A. Peyrot,  
`On large values of $L(\sigma,\chi)$', {\it  Quart. J. Math.\/}, {\bf 70} (2019), 831--848. 

 
\bibitem{ArCr} L.-P. Arguin and N.  Creighton, 
`Upper bounds on large deviations of Dirichlet $L$-functions in the $q$-aspect',
 {\it J. Number Theory\/}, {\bf 273} (2025) 96--158.
 
 

 
\bibitem{BKT} 
A.  Bainbridge, R. Khan and Z.  S. Tang, 
`On moments of $L$-functions over Dirichlet characters', 
{\it Can. Math. Bull.\/},  {\bf  69} (2026), 490--513.


\bibitem{BFKMM1} V. Blomer, {\'E}. Fouvry, E. Kowalski, P. Michel and D. Mili{\'c}evi{\'c},
 `On moments of twisted $L$-functions',
 {\it Amer. J.  Math.\/}, {\bf 139} (2017), 707--768.


\bibitem{BFKMM2} V. Blomer, {\'E}. Fouvry, E. Kowalski, P. Michel and D. Mili{\'c}evi{\'c},
 `Some applications of smooth bilinear forms with Kloosterman sums',
  {\it Proc. Steklov Math. Inst.\/},  {\bf 296} (2017), 18--29.
  
\bibitem{BoSe} A.  Bondarenko and K. Seip,  
`Extreme values of the Riemann zeta function and its argument',
  {\it Math. Ann.\/},  {\bf 372} (2018),   999--1015. 
  
 
    \bibitem{BKS} J. Bourgain, S. V. Konyagin and I. E. Shparlinski, `Product sets of rationals, multiplicative translates of subgroups in residue rings, and fixed points of the discrete logarithm', {\it Int. Math. Res. Not.\/}, {\bf 2008} (2008), Article ID rnn090.
 

  \bibitem{CillGar} J. Cilleruelo and M. Z. Garaev, 
  `Congruences involving product of intervals and sets with small multiplicative doubling modulo a prime and applications', 
  {\it Math. Proc. Camb. Phil. Soc.\/}, {\bf 160} (2016), 477--494. 
  
  
  \bibitem{CoGoZa} C. Cobeli, S.  Gonek and
  A. Zaharescu, `On the distribution
  of small powers  of a primitive root',
  {\it J. Number Theory\/}, {\bf 88} (2001),  49--58.

\bibitem{DDLL} P.  Darbar,  C. David, M. Lalin and A.  Lumley, 
`Asymmetric distribution of extreme values of cubic $L$-functions at 
$s=1$', {\it J. Lond. Math. Soc.\/}, {\bf 110} (2024), Art.~12996.

\bibitem{DM2025} P. Darbar and G. Maiti, `Large values of quadratic Dirichlet $L$-functions",  {\it Mathematische Annalen\/}, {\bf 392} (2025), 4573--4605.

\bibitem{Dav} H. Davenport, {\it Multiplicative Number Theory\/},
 Springer, New York, 2000.

\bibitem{dlbTen} 
R. de la Bret\`{e}che and G. Tenenbaum, `Sommes de G\'{a}l et applications', 
{\it Proc. Lond. Math. Soc.\/}, {\bf 119} (2019), 104--134.

 \bibitem{DWZ} Z. Dong,  W. Wang and H. Zhang,
`Distribution of Dirichlet $L$-functions', 
{\it Mathematika\/}, {\bf  69} (2023),  719--750. 


\bibitem{Erm} I. Ermoshin,
`Large central values of Dirichlet $L$-functions in cosets', 
{\it Preprint \/}, 2025, available from \url{https://arxiv.org/abs/2504.05890}.



\bibitem{FKM} 
{\'E}. Fouvry, E. Kowalski and Ph. Michel,
`Toroidal families and averages of $L$-functions,~I',
{\it Acta Arith.\/}, {\bf  214} (2024), 109--142. 
 

\bibitem{GaYo} B. Garcia  and M. P. Young, 
`Asymptotic second moment of Dirichlet $L$-functions along a thin coset',  
{\it Forum  Math., Sigma\/}, {\bf  13} (2025),   Art.~e83., 1--19.

\bibitem{GrRu}
B. Green and I. Ruzsa, `Counting sumsets and sum-free sets modulo a prime',  
{\it Studia Sci. Math. Hungar.\/}, {\bf  41} (2004), 285--293.


\bibitem{GrSo-0}
A. Granville and K. Soundararajan, 
`Large character sums', 
{\it J. Amer. Math. Soc.\/}, {\bf 14} (2001), 365--397.

 
\bibitem{GrSo-2}
A. Granville and K. Soundararajan, 
`The distribution of values of $L(1, \chi_d)$', {\it Geom. Funct. Anal.\/}, {\bf 13} (2003), 992--1028.


\bibitem{GrSo-3} A. Granville and K. Soundararajan,
`Extreme values of $\vert \zeta(1 + it)\vert$', {\it  The Riemann
zeta function and related themes: Papers in honour of Professor K. Ramachandra},
Ramanujan Math. Soc. Lect. Notes Ser., vol.~2 , 2006, 65--80.


\bibitem{HeLi} W. Heap and J. Li, `Simultaneous extreme values of zeta and $L$-functions',  {\it Mathematische Annalen\/}, {\bf 390} (2024), 6355--6397.

\bibitem{H-B} D. R. Heath-Brown, `A large values estimate for Dirichlet
polynomials',  {\it J. London Math. Soc.\/}, {\bf  20} (1979), 8--18.

\bibitem{IwKow} H. Iwaniec and E. Kowalski,
{\it Analytic number theory\/}, Amer.  Math.  Soc.,
Providence, RI, 2004. 


\bibitem{KLM} 
R. Khan, R. Lei and D. Milic{\'e}vi{\'c}, `Effective moments of Dirichlet $L$-functions in Galois orbits', 
 {\it Int. Math. Res. Not.\/}, {\bf 12} (2019), 475--490.   
 
\bibitem{KMN1} 
R. Khan, D. Milic{\'e}vi{\'c} and H. T. Ngo, `Non-vanishing of Dirichlet $L$-functions in Galois orbits', 
 {\it Int. Math. Res. Not.\/}, {\bf 2016} (2016), 6955--6978. 
 
 \bibitem{KoSh} S. V. Konyagin and  I. E.  Shparlinski,
 {\it Character sums with  exponential functions and their applications\/}, Cambridge Univ. Press, Cambridge, 1999.  

 \bibitem{KMN2} 
R. Khan, D. Milic{\'e}vi{\'c} and H. T. Ngo, `Non-vanishing of Dirichlet $L$-functions, II', {\it Math. Z.\/}, {\bf  300} (2022), 1603--1613. 

\bibitem{Kouk}
D. Koukoulopoulos, `Pretentious multiplicative functions and the prime number 
theorem for arithmetic progressions', {\it Compositio Math.\/},  {\bf 149} (2013), 1129--1149.

\bibitem{Lam0} Y. Lamzouri, `The Two-Dimensional Distribution of Values of $\zeta(1+it)$, {\it Int. Math. Res. Not.\/}, {\bf 2008} (2008), Article ID rnn~106. 

\bibitem{Lam1} Y. Lamzouri, `On the distribution of extreme values of zeta and $L$-functions in the strip
$1/2<\sigma<1$', {\it  Int. Math. Res. Not.\/}, {\bf 2011} (2011), 5449--5503. 
 
 
 \bibitem{Lam2} Y. Lamzouri, `Large values of $L(1,\chi)$ for $k$-th order characters $\chi$ and applications to character sums',  {\it Mathematika\/}, {\bf 63} (2017), 53--71.  
 
 \bibitem{LeLe}
S. H. Lee and S. Lee,
`The twisted moments and distribution of values of Dirichlet $L$-functions at $1$',
 {\it J. Number Theory\/}, {\bf 197} (2019), 168--184.
 
  \bibitem{Lou1} S. Louboutin, `Dedekind sums and mean square value of $L(1, \chi)$ over subgroups', 
{\it Publ. Math. Debrecen\/}, {\bf 106} (2025), 387--403. 


 \bibitem{Lou2} S. Louboutin, `An asymptotic on the logarithms of the relative
class numbers of imaginary abelian number fields of
prime conductor and large degree', 
{\it Preprint\/}, 2025, available  from \url{https://arxiv.org/abs/2501.18941}. 

\bibitem{LoMu1} S. Louboutin and M. Munsch,  `Second moment of Dirichlet $L$-functions, character sums over subgroups and upper bounds on relative class numbers', {\it Quart. J. Math.\/}, {\bf 72} (2021), 1379--1399.

\bibitem{LoMu2} S. Louboutin and M. Munsch, `Mean square values of $L$-functions over subgroups for non primitive characters, 
Dedekind sums and bounds on relative class numbers', {\it Canad. J. Math.\/}, {\bf 75} (2023), 1711--1743. 

\bibitem{Mont}
H.~L. Montgomery, {\it Topics in multiplicative number theory\/}, 
Lecture Notes in Math., Vol.~227, Springer-Verlag, Berlin-New
  York, 1971.
  
  \bibitem{Mont2} 
  H. L. Montgomery, `Distribution of small powers of a
  primitive root', {\it Advances in Number Theory\/}, 
  Clarendon Press, Oxford, 1993, 137--149.
  
  \bibitem{MunShp}
M. Munsch and  I. E. Shparlinski, 
`Moments and non-vanishing of $L$-functions over thin subgroups', 
{\it Forum Math.\/},  (to appear).  

 \bibitem{RuZa} Z. Rudnick and A. Zaharescu, `The distribution of
spacings between small powers of a primitive root',
{\it Israel J. Math.\/}, {\bf  120} (2000), 271--287.

\bibitem{Shp} I. E. Shparlinski, 
`Ratios of small integers in multiplicative subgroups of residue rings',  
{\it Exp. Math.\/}, {\bf 25} (2016),  273--280.

\bibitem{ShpYau} I. E. Shparlinski and K.-H. Yau, 
`Bounds of double multiplicative character sums and gaps between residues of exponential functions''
{\it J. Number Theory\/}, {\bf 167} (2016), 304--316. 

\bibitem{Sound} K. Soundararajan, `The fourth moment of  Dirichlet $L$-functions', 
{\it Analytic Number Theory\/}, Clay Math. Proc.~7, Amer. Math. Soc., Providence, RI, 2007,  239--246.


\bibitem{TaoVu}
T. Tao and V. Vu, \textit{Additive combinatorics}, Stud. Adv. Math. 105,
Cambridge Univ. Press, Cambridge, 2006.

\bibitem{VaZa}   M. V{\^ a}j{\^ a}itu and A. Zaharescu,
`Differences between powers of a primitive root', 
{\it  Internat. J. Math. Math. Sci.\/}, {\bf 29} (2002), 325--331.


\bibitem{Wu1}
X. Wu,
`The fourth moment of Dirichlet $L$-functions at the central value',
{\it Math. Ann.\/}, {\bf 387}  (2023), 1199--1248.

\bibitem{Wu2}
X. Wu,
`The fourth moment of Dirichlet $L$-functions along the critical line', 
{\it Forum Math.\/},  {\bf 35} (2023),  1347--1371. 

\end{thebibliography}
\end{document}